\documentclass[letterpaper,11pt]{article}
\usepackage[utf8]{inputenc}
\usepackage[english]{babel}

\usepackage{latexsym}
\usepackage{amssymb}
\usepackage{amsthm}
\usepackage{mathrsfs}
\usepackage{amstext}
\usepackage{graphicx}
\usepackage{amsfonts}
\usepackage{wrapfig}
\usepackage{sidecap}
\usepackage{latexsym}
\usepackage{pdfpages}
\usepackage{epic}
\usepackage{fullpage}
\usepackage{color}
\usepackage{curves}
\usepackage{subfigure}
\usepackage{setspace}
\usepackage{amsmath}
\numberwithin{equation}{section}
\usepackage{anysize}
\usepackage{rotating}
\usepackage{enumerate} 
\usepackage{hyperref}
\usepackage{bbm}
\usepackage{subfigure}

\newtheorem{theorem}{Theorem}[section]
\newtheorem{definition}[theorem]{Definition}
\newtheorem{corollary}[theorem]{Corollary}
\newtheorem{proposition}[theorem]{Proposition}

\newtheorem{remark}[theorem]{Remark}
\def\R{{\mathbb R}}

\def\C{{\mathbb C}}

\renewcommand{\leq}{\leqslant}
\renewcommand{\geq}{\geqslant}
\numberwithin{equation}{section}

\newcommand{\pt}{\partial_t}
\newcommand{\ptt}{\partial_t^2}

\newcommand{\pnu}{\partial_\nu}

\newcommand{\IOT}[1]{\int_{0}^T\int_\Omega #1 \,dxdt} 
\newcommand{\IGT}[1]{\int_{0}^T\int_{\Gamma_1} #1 \,d\sigma dt}
\newcommand{\IGOT}[1]{\int_{0}^T\int_{\Gamma_0} #1 \,d\sigma dt}
\newcommand{\IGST}[1]{\int_{0}^T\int_{\Gamma_*} #1 \,d\sigma dt}  

\newcommand{\RM}[1]{\textcolor{blue}{#1}}

\title{Exact Controllability for a Schr\"odinger equation with dynamic boundary conditions\footnote{The first author ... The second author has been supported by FONDECYT 3200830}}
\author{
	Alberto Mercado\thanks{Departamento de Matem\'atica, Universidad T\'ecnica Federico Santa Mar\'{\i}a, Casilla 110-V, Valpara\'{\i}so, Chile e-mail: alberto.mercado@usm.cl}
	\and 
	Roberto Morales\thanks{Departamento de Matem\'atica, Universidad T\'ecnica Federico Santa Mar\'{\i}a, Casilla 110-V, Valpara\'{\i}so, Chile e-mail: roberto.moralesp@usm.cl}
	}

\date{\today}
\begin{document}
\maketitle 
\begin{abstract}
In this paper, we study the controllability of a Schr\"odinger equation  with mixed boundary conditions on disjoint subsets of
the boundary: dynamic boundary condition of Wentzell type, and Dirichlet boundary condition. The main result of this article is given by 
new Carleman estimates for the associated adjoint system, where the weight function is constructed specially adapted to the geometry of the domain. Using these estimates, we  prove the exact controllability of the system with a boundary control acting only in the part of the boundary where the Dirichlet condition is imposed. Also, we obtain  a distributed exact controllability result for the system.
\end{abstract}


{\bf AMS} 35Q41, 93B05, 93B07, 93C05, 35M13.

\section{Introduction}
In this article, controllability properties of a non-conservative Schr\"odinger equation with dynamic boundary conditions of Wentzell type are studied.
Let $\Omega\subset \R^n$, $n\geq 2$, be a bounded domain with regular boundary $\partial \Omega$ 
such that $\partial \Omega=\Gamma_0\cup \Gamma_1$ with $\Gamma_0$ and $\Gamma_1$ are two closed subsets and 
$\Gamma_0\cap \Gamma_1=\emptyset$;
a typical example is given by
 the annulus  $\Omega = \{ x \in \R^n \, : \, R_1 < |x| < R_2 \}$.
Then we consider  the  system given by
\begin{align}
	\label{intro:problem:01}
	\begin{cases}
		i\pt y+d\Delta y-\vec{q}_1\cdot \nabla y + q_0y=0,&\text{ in }\Omega\times (0,T),\\
		i\pt y_\Gamma -d\pnu y +\delta \Delta_\Gamma y_\Gamma -\vec{q}_{\Gamma,1}\cdot \nabla_\Gamma y_\Gamma+ q_{\Gamma,0} y_\Gamma=0,&\text{ on }\Gamma_1 \times (0,T),\\
		y=y_\Gamma,&\text{ on }\Gamma_1 \times (0,T),\\
		y=\mathbbm{1}_{\Gamma_*} h,&\text{ on }\Gamma_0 \times (0,T),\\
		(y(0),y_\Gamma (0))=(y_0,y_{\Gamma,0}),&\text{ in }\Omega\times \Gamma_1,
	\end{cases}
\end{align}
where $h\in L^2(\Gamma_* \times (0,T);\mathbb{C})$ (with $\Gamma_*\subseteq \Gamma_0$) is a  control acting on a subset of $\Gamma_0\times (0,T)$. Besides, $d>0$ and $\delta>0$ are parameters representing  the diffusion on the bulk and on the boundary, respectively. Moreover, $(\vec{q}_1,\vec{q}_{\Gamma,1})\in [L^{\infty}(\Omega\times (0,T);\mathbb{C})]^n\times [L^{\infty}(\Gamma_1\times (0,T);\mathbb{C})]^n$ and $(q_0,q_{\Gamma,0})\in L^{\infty}(\Omega\times (0,T);\mathbb{C}) \times L^\infty(\Gamma_1\times (0,T);\mathbb{C})$ are lower order potentials. We denote by $\Delta_\Gamma$ the Laplace-Beltrami operator, $\nabla_\Gamma$ is the tangential gradient and $\pnu y$ the normal derivative associated to the outward normal $\nu$ of $\Omega$.
We wish to investigate controllability properties of the system \eqref{intro:problem:01}. Roughly speaking, we are interested in finding conditions on the geometry of the domain and the parameters of the system such that the associated solution $(y,y_{\Gamma})$ can be driven to any given  state at time $T>0$. More precisely, we study the exact controllability of \eqref{intro:problem:01}, which can be defined as follows:
\begin{definition}
	System \eqref{intro:problem:01} is said to be \textbf{exactly controllable} at time $T>0$ 
	in space $X$ 
	if for every states $(y_0,y_{\Gamma,0}),(y_T,y_{\Gamma,T})\in 
	X $, there exists a (boundary) control $h\in L^2(\RM{\Gamma_*}\times (0,T);\mathbb{C})$ such that the associated solution $(y,y_\Gamma)$ satisfies
	\begin{align*}
	(y(T),y_{\Gamma}(T))=(y_T,y_{\Gamma,T}),\text{ in }\Omega\times \Gamma_1.
	\end{align*}
\end{definition}

We point out that the control $h$ acts  only on a portion of the boundary $\Gamma_0$. This means that the equation in the bulk is controlled directly by $h$, while the equation on the boundary $\Gamma_1$ is being controlled indirectly through the side condition $y=y_\Gamma$ on $\Gamma_1\times (0,T)$.

Linear and nonlinear Schr\"odinger equations have been intensely studied due to their applications to plasma physics and laser optics, see e.g. \cite{bandrauk1993molecules} and \cite{giusti1992vibrational}. On the other hand, Schr\"odinger equation can be represented as two suitable diffusion equations which their solutions are in duality, see \cite{nagasawa2012schrodinger}. In our case, the system \eqref{intro:problem:01} can  model the evolution of diffusion processes in an object $\Omega$ and its interaction with 
 another one, having a thick border $\Gamma_1$.

Controllability properties of the Schr\"odinger equation has been studied by several authors in the last 30 years. In \cite{lebeau1992controle}, G. Lebeau proved that the Geometric Control Condition (GCC) for the exact controllability of the wave equation is sufficient for the exact controllability of the Schr\"odinger equation in any time $T>0$. The proof of this result is based on the diadic decomposition of the Fourier representation of solutions of the Schr\"odinger equation which allows viewing them as superposition of an infinite sequence of solutions of wave equations with velocity of propagation tending to infinity. Besides, a particular case of this work is due to E. Machtyngier in \cite{machtyngier1994exact}. In this case, the exact controllability in $H^{-1}(\Omega)$ with $L^2$-boundary control and exact controllability in $L^2(\Omega)$ with $L^2$-controls supported in a neighborhood of the boundary are achieved. These results were obtained using multiplier techniques,  Hilbert Uniqueness Method (HUM) and Holmgren unique continuation principle, a property which   relies on the analyticity of the coefficients.

On the other hand, contrary to the results of hyperbolic equations, there are relevant controllability results for the Schr\"odinger equation in some situations in which the GCC is not fulfilled in any time $T$. We refer to \cite{jaffard1988controle} and \cite{burq1993controle} where the main results are based 
on the decomposition of  the Plate operator $\ptt +\Delta^2$ in  two conjugate Schr\"odinger operators in the following form:
\begin{align*}
\ptt u +\Delta^2u=(i\pt +\Delta)(-i\pt +\Delta)u.
\end{align*}

A useful tool to obtain observability inequalities is given by  the so-called Carleman estimates. In the case of Schr\"odinger equation with Dirichlet boundary conditions, several authors established  controllability and stability results using these estimates combining another techniques. In \cite{baudouin2002uniqueness}, the authors derived a Carleman estimate under a strict pseudoconvexity condition, or equivalently, under a strong convexity of the weight function in the space variable. We also mention that the observation is taken in regions according to the classical geometric conditions typically used for the wave equation. As a consequence of this result, the authors proved a Lipschitz stability estimate for an inverse problem for the potential of the Schr\"odinger operator in $H^1(\Omega)$. On the other hand, in \cite{mercado2008inverse} the authors proved a more general Carleman estimate replacing the strong pseudoconvexity condition for a weaker one, 
where the Hessian of the weight function may be degenerate at some points, allowing to consider weights of the form $\psi(x)=x\cdot e$, where $e\in \mathbb{R}^n$, resulting in  observation regions not satisfying 
the geometric control conditions. However, in these estimates only a part of the weighted-$H^1$ energy may be bounded by boundary/internal observations. For other results concerning controllability of the Schr\"odinger equation with Dirichlet boundary conditions we refer to \cite{phung2001observability}, \cite{aassila2003exact}, \cite{rosier2009exact}, \cite{rosier2009null} and \cite{zuazua2003remarks}.

Controllability properties of PDEs with dynamic boundary conditions have  also been intensively studied in the last  years. We mention  the works \cite{maniar2017null} and \cite{khoutaibi2020null} where the authors studied controllability properties of parabolic equations with these kind of boundary conditions. Using the approach of Fursikov and Imanuvilov and using the fact that the tangential derivatives of the associated weight function vanish, the authors determined the null controllability of such systems with arbitrary small regions. Based on these results, in \cite{zhang2019insensitizing} it is  studied the existence of insensitizing controls for systems of parabolic equations with Wentzell boundary conditions. We also refer the works \cite{ismailov2018inverse} and \cite{ismailov2019inverse} where some inverse problems for such models are considered.

Controllability properties in the  case of the wave equation with dynamic boundary conditions were also recently studied. In \cite{gal2017carleman}, the authors studied a  non-conservative  wave system acting in a domain $\Omega$ with the same topology structure that the case studied here. 
They  proved a global Carleman estimate with observations on both sides of the boundary, i.e., on the usual observation region satisfying the geometric control condition and on the part of the boundary where the dynamic conditions are given; as a consequence, it is obtained a controllability result  with two controls. On the other hand, in \cite{benabbas2018observability}, the authors determined that, in the particular situation when $\Omega$ is an $n-$dimensional interval,
the control acting on the subset of the boundary where  the Wentzell boundary conditions 
are imposed can be neglected. This result is based on Fourier expansions and a generalization of the Ingham inequalities due to Mehrenger for $n$-dimensional intervals.

Recently, some uniform stability results on the solutions for the Schr\"odinger and Ginzburg-Landau equations with dynamic boundary conditions were obtained in \cite{cavalcanti2016well} and \cite{correa2018complex}, respectively. We also mention the articles \cite{Lasiecka01} and \cite{Lasiecka02} where the authors studied Carleman estimates in $H^1$ and $L^2$, respectively, for non-conservative Schr\"odinger equations with different types of boundary conditions.
 However, to the best of the authors' knowledge, this is the first time that controllability properties of the Schr\"odinger equation with Wentzell boundary conditions like \eqref{intro:problem:01} are studied.

\subsection{General setting}
In this section, we set up the notation and terminology used in this paper. 
We consider the set  $\Gamma_1\subset \partial \Omega$ as an $(n-1)$-dimensional compact Riemannian submanifold equipped by the Riemannian metric $g$, induced by the natural embedding $\Gamma_1\subset \mathbb{R}^n$. It is possible to define the differential operators on $\Gamma_1$ in terms of the Riemannian metric $g$. 
However, for the purposes of this article, it will be enough to use the 
most important properties of the underlaying  operators and spaces.  The details can be found, for instance, in  \cite{Jost} and \cite{Taylor}. For the sake of completeness, we recall some of those properties.

The tangential gradient $\nabla_\Gamma$ of $y_\Gamma$ at each point $x \in \Gamma_1$ can be seen as the  projection of the standard Euclidean gradient $\nabla y$ onto the tangent space of $\Gamma_1$ at $x\in \Gamma_1$,
where $y_\Gamma$ is the trace of $y$ on $\Gamma_1$. That is to say, 
$$\nabla_\Gamma y_\Gamma=\nabla y- \nu \pnu y,$$
where $y =y_\Gamma$ on $\Gamma_1$ and $\pnu y$ is the normal derivative associated to the outward normal $\nu$. In this way,  the tangential divergence $\text{div}_\Gamma$ in  $\Gamma_1$ is defined by
\begin{align*}
\text{div}_\Gamma (F_\Gamma):H^1(\Gamma_1;\mathbb{R})\to \mathbb{R},\quad y_\Gamma \mapsto -\int_{\Gamma_1} F_\Gamma \cdot \nabla_\Gamma y_\Gamma\, d\sigma.
\end{align*} 

The Laplace Beltrami operator is given by 
$\Delta_\Gamma y_\Gamma =\text{div}_\Gamma (\nabla_\Gamma y_\Gamma)$ for all $y_\Gamma \in H^2(\Gamma_1;\mathbb{R})$. In particular, the surface divergence theorem holds:
\begin{align*}
\int_{\Gamma_1} \Delta_\Gamma yz\, d\sigma=-\int_{\Gamma_1} \nabla_\Gamma y\cdot \nabla_\Gamma z\, d\sigma,\quad \forall y\in H^2(\Gamma_1;\mathbb{R}),\quad \forall z\in H^1(\Gamma_1;\mathbb{R}).
\end{align*}

In order to simplify the notation, here and subsequently, the function spaces refer to complex-valued functions unless otherwise stated.

We introduce the Hilbert space $\mathcal{H}=L^2(\Omega)\times L^2(\Gamma_1)$  in $\mathbb{C}$ equipped with the
scalar product
\begin{align*}
\langle (u,u_\Gamma),(v,v_\Gamma) \rangle_{\mathcal{H}}= \int_\Omega u\overline{v} \,dx+ \int_{\Gamma_1} u_\Gamma \overline{v_\Gamma}\, d\sigma.
\end{align*}

Moreover, for $m\in \mathbb{N}$, we consider the space
\begin{align*}
H_{\Gamma_0}^m(\Omega)=\{u\in H^m(\Omega)\,:\, u=0\text{ on }\Gamma_0\},
\end{align*}
which is a closed subspace of the Sobolev space $H^m(\Omega)$. In the same manner, we define the space
\begin{align*}
\mathcal{V}^m =\{(u,u_\Gamma)\in H_{\Gamma_0}^m(\Omega)\times H^m(\Gamma_1)\,:\, u=u_\Gamma\text{ on }\Gamma_1\}.
\end{align*}
and by simplicity we write $\mathcal{V}=\mathcal{V}^1$. Due to the Poincar\'e inequality and a trace theorem,
we have
	\begin{align*}
		\int_\Omega |u|^2 dx + \int_{\Gamma_1} |u_\Gamma|^2 d\sigma \leq C \int_\Omega |\nabla u|^2 dx,
	\end{align*}
for all $u \in \mathcal{V}$, and then we deduce that $\mathcal{V}$ is a Hilbert space in $\C$ with the inner product given by
\begin{align*}
\langle (u,u_\Gamma),(v,v_\Gamma) \rangle_{\mathcal{V}}=\int_\Omega \nabla u \cdot \nabla \overline{v}\,dx + \int_{\Gamma_1} \nabla_\Gamma u_\Gamma \cdot \nabla_\Gamma \overline{v}_\Gamma\, d\sigma.
\end{align*}

Now, we present a definition related with the 
geometric hypothesis we will assume for  the interior boundary.
\begin{definition} 
	An open, bounded and convex set $U\subset \mathbb{R}^n$, is said to be {\bf strongly convex} if $\partial U$ is of class $C^2$ and all the principal curvatures are strictly positive functions on $\partial U$.
\end{definition}
We point out that $U\subset \mathbb{R}^n$ is strongly convex if and only if for all plane $\Pi \subset \mathbb{R}^n$ intersecting $U$, the curve $\Pi \cap \partial U$ has strictly  positive curvature at each point. In particular, a strongly convex set is geometrically strictly convex, 
in the sense that it has, at each of its boundary points, a supporting hyperplane with exactly one contact point.

We assume that
$\Omega = \Omega_0 \setminus \overline \Omega_1$, where $\Omega_1$
is   strongly convex.  Also, without loss of generality, we suppose that $0 \in \Omega_1$ (if it is not the case, we can 
take $x_0\in \Omega_1$ and then perform a translation by $-x_0$). Then, we set  $\Gamma_{k} = \partial \Omega_k$ for $k=0,1$.

Now we can define the Carleman weight function we will  use in this work. For each $x \in \R^n$, we set
\begin{equation} \label{Mink}
\mu(x) = \inf \{ \lambda > 0 \, : \, x \in \lambda \Omega_1\}.
\end{equation}

We define for each $x \in \Omega$, $\psi(x)=\mu^2(x)$, and for $\lambda>0$ we set 
\begin{align}
\label{def:weight:functions}
\theta(x,t)=\dfrac{e^{\lambda \psi(x)}}{t(T-t)},\quad \varphi(x,t)=\dfrac{\alpha -e^{\lambda \psi(x)}}{t(T-t)},\quad  \forall (x,t)\in {\Omega}\times (0,T), 
\end{align}
where $\alpha>\|e^{\lambda \psi}\|_{L^\infty(\Omega)}$.


In order to guarantee enough regularity of the  function $\mu$, we will assume that  $\partial \Omega_1$
has a regular parametrization. In order to be explicit, we
will ask the following property. 
\begin{equation} \label{regmu}
\text{ The bijection } {\Phi:}\,{x \in \partial \Omega_1}\mapsto {\frac{x}{\|x\|} \in S^{n-1}}
\text{ is a }  C^4 \text{ diffeomorphism.}
\end{equation}
We recall that  $\Phi$ is well-defined and is a bijective function thanks to the fact that $\Omega_1$ is convex and it contains the origin.

We point out that the weight functions defined in \eqref{def:weight:functions} have been previously used to deduce Carleman estimates for transmission problems for wave and Schr\"odinger equations  with Dirichlet boundary conditions, see \cite{BMO} and \cite{baudouin2008inverse}. However, to the author's knowledge it is the first time that this function is used to deduce controllability results for Schr\"odinger equation with dynamic boundary conditions.

\begin{remark} 
	The function $\mu$ defined in \eqref{Mink} is called the Minkowski functional of the set $\Omega_1$.
	By definition, given $x \in \Omega$, if $\lambda = \lambda(x) > 0$ is such that $ x \in \lambda \partial \Omega_1 = \lambda \Gamma_1$ (by convexity, there exists 
	exactly  only such $\lambda$), then 
	$\mu(x) = \lambda$.
	This function has the well-known property of, under adequate hypothesis on the open set $\Omega_1$,  defining 
	a norm in $\R^n$ such that their unit ball is $\Omega_1$. 
\end{remark}

\subsection{Main results}

In this section, we give the main results of this article.
The first result is a Carleman estimate for a Schr\"odinger equation with dynamic boundary conditions, whose proof is given in Section \ref{section:proof:carleman}.

\begin{theorem}
	\label{Thm:Carleman}
	Suppose that 
	$\Omega = \Omega_0 \setminus \overline \Omega_1$, where
	$\Omega_0$ is an open bounded set with $C^2$ boundary
	and $\Omega_1$ is an open strongly convex set 
	with boundary $\partial \Omega_1$ satisfying regularity hypothesis \eqref{regmu}.
	Let $(\vec{q}_1,\vec{q}_{\Gamma,1})\in [L^\infty(\Omega\times (0,T))]^n\times [L^\infty(\Gamma_1 \times (0,T))]^{n}$ and $(q_0,q_{\Gamma,0})\in L^\infty(\Omega\times (0,T))\times L^\infty(\Gamma_1\times (0,T))$. Also, assume that $\delta$ and $d$  are  positive constants satisfying
	\begin{align}
	\label{assumption:delta:d}
	\delta > d.
	\end{align}
	Then, there exist constants $C$, $s_0$ and $\lambda_0$ such that
	\begin{align}
	\label{Carleman:estimate:01}
	\begin{split}
	&\IOT{e^{-2s\varphi} \left(s^3\lambda^4\theta^3|v|^2 +s\lambda \theta |\nabla v|^2 +s\lambda^2 \theta |\nabla \psi \cdot \nabla v|^2 \right) }\\
	&+\IGT{e^{-2s\varphi} \left(s^3\lambda^3 \theta^3 |v_\Gamma|^2 + s\lambda \theta|\pnu v|^2 + s\lambda \theta |\nabla_\Gamma v_\Gamma|^2 \right)}\\
	\leq & C\IOT{e^{-2s\varphi}|L(v)|^2}
	+ C \IGT{e^{-2s\varphi} |N(v,v_\Gamma)|^2}\\
	&+C s\lambda \IGST{e^{-2s\varphi} \theta|\pnu v|^2},
	\end{split}
	\end{align}   
	for all 
	$ \lambda\geq \lambda_0$,  $ s\geq s_0$ and 
	$(v,v_\Gamma)\in L^2(0,T;\mathcal{V})$ where
	\begin{align*}
	L(v):=&i\pt v+d\Delta v + \vec{q}_1\cdot \nabla v +q_0v\in L^2(\Omega\times (0,T)),\\
	N(v,v_\Gamma):=&i\pt v -d\pnu v + \delta \Delta_\Gamma v_\Gamma+ \vec{q}_{\Gamma,1}\cdot \nabla_\Gamma u_\Gamma + q_{\Gamma,0} v_\Gamma\in L^2(\Gamma_1\times (0,T))
	\end{align*}
	and 
	$\pnu v \in L^2(\partial \Omega\times (0,T))$, with 
	\begin{align}
	\label{def:Gamma-ast} 
	\Gamma_* := \{x\in \partial \Omega\,:\, \pnu \psi(x)\geq 0\} \subseteq \Gamma_0.
	\end{align}
\end{theorem}

\begin{remark} 
	The hypothesis about  the convexity  of $\Omega_1$ allows us to define a weight function $\psi$ adapted to the geometry of our problem (see \eqref{Mink} and \eqref{def:weight:functions}): it can be used as a Carleman weight function, and has the 
	particularity that it is constant on $\Gamma_1$, which implies that $\nabla_\Gamma \psi = 0$ and $\pnu \psi < 0$ in $\Gamma_1$.  
\end{remark}
\begin{remark}
	
	Hypothesis \eqref{assumption:delta:d} is used in order to estimate the term with $\nabla _\Gamma v_\Gamma$ on $\Gamma_1\times (0,T)$ in the left-hand side of the Carleman estimate \eqref{Carleman:estimate:01} (see inequality \eqref{conclu:01}). 
\end{remark}
\begin{remark} 
	We recall that $\partial \Omega=\Gamma_0 \cup \Gamma_1$, where $\Gamma_0$ and $\Gamma_1$ are connected and closed. In Figure \ref{pic:geo}, we illustrate an example of the shape of $\Omega$ under the geometrical assumption of Theorem \ref{Thm:Carleman}.
\end{remark}

\begin{figure}[!h]
	\begin{center}
		\includegraphics[scale=0.5]{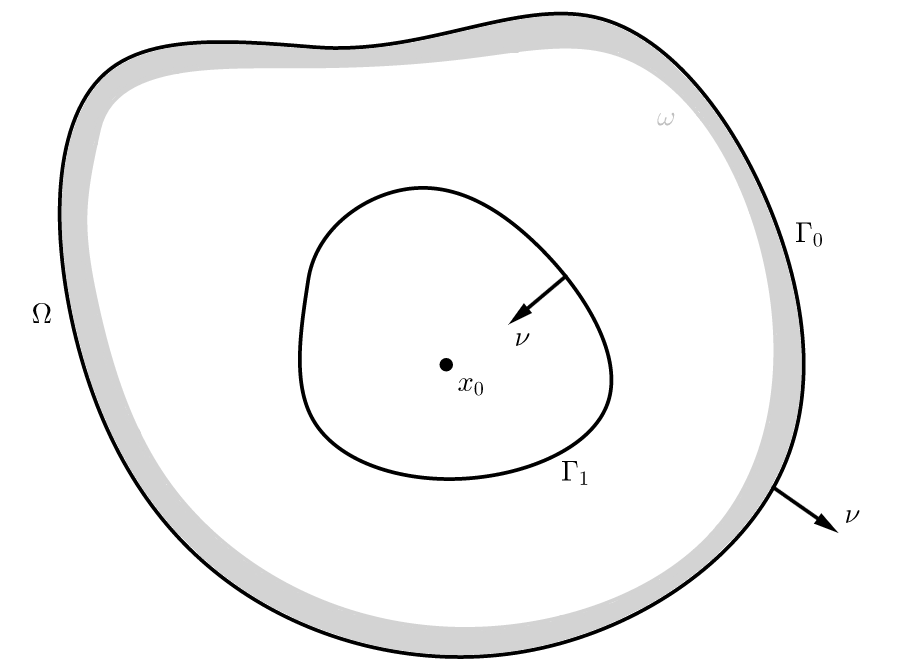}
		\caption{Geometric assumptions of Theorem \ref{Thm:Carleman}.}
		\label{pic:geo}
	\end{center}
\end{figure}
As a direct consequence of the Theorem \ref{Thm:Carleman}, we can obtain a Carleman estimate where the observation is taken in a boundary neighborhood of $\Gamma_*$ {(see for example the region $\omega$ in gray in Figure \ref{pic:geo})}.

\begin{corollary}
	\label{Corollary:Carleman}	
	Let $\omega\subset \Omega$ be an open set such that there exists $\varepsilon>0$ such that 
	\begin{align*}
	\omega \supset \{ x\in \Omega\,:\, \text{dist}(x,\Gamma_\ast)\leq \varepsilon\},
	\end{align*}
	where $\Gamma_*$ is given by \eqref{def:Gamma-ast}. Then, under the assumptions of Theorem \ref{Thm:Carleman}, there exist  positive constants $C,s_0,\lambda_0$ such that
	\begin{align}
	\label{Carleman:estimate:distributed}
	\begin{split}
	&\IOT{e^{-2s\varphi}(s^3\lambda^4 \theta^3 |v|^2 + s\lambda \theta |\nabla v|^2)}\\
	&+\IGT{e^{-2s\varphi}(s^3\lambda^3 \theta^3 |v|^2 +s\lambda \theta |\nabla_\Gamma v_\Gamma|^2 + s\lambda \theta |\pnu v|^2)}\\
	\leq & C\IOT{e^{-2s\varphi} |L(v)|^2}+C\IGT{e^{-2s\varphi} |N(v,v_\Gamma)|^2}\\
	&+C \int_0^T\int_\omega e^{-2s\varphi} \left( s^3\lambda^4 \theta^3 |v|^2 +s\lambda \theta |\nabla v|^2 \right) dx dt 
	\end{split}
	\end{align}	
	for all $ \lambda \geq \lambda_0$ and $ s\geq s_0$, and 
	for all pair $(v,v_\Gamma)\in L^2(0,T;\mathcal{V})$ satisfying $L(v)\in L^2(\Omega\times (0,T))$, $N(v,v_\Gamma)\in L^2(\Gamma_1\times (0,T))$
	and  $\pnu v \in L^2(\partial \Omega\times (0,T))$.
\end{corollary}

With Theorem \ref{Thm:Carleman} and Corollary \ref{Corollary:Carleman} at hand, we derive exact controllability results for Schr\"odinger equations with dynamic boundary conditions. To this end, we suppose that
\begin{align}
\label{hip:01} 
\vec{q}_1=d\nabla \pi-i\vec{r}, \quad \vec{q}_{\Gamma,1}=\delta\nabla_\Gamma \pi_\Gamma-i\vec{r}_\Gamma,
\end{align}
where 
$$\pi\in L^{\infty}(0,T;W^{3,\infty}(\Omega;\mathbb{R}))
\cap  W^{1,\infty}(0,T;W^{1,\infty}(\Omega;\mathbb{R})),$$
 $$\pi_\Gamma \in L^\infty(0,T;W^{3,\infty}(\Gamma_1;\mathbb{R})) \cap  W^{1,\infty}(0,T;W^{1,\infty}(\Gamma_1;\mathbb{R}))$$ 
 with $\pi=\pi_\Gamma$ on $\Gamma_1\times (0,T)$, and 
 $$\vec{r}\in [L^\infty(0,T;W^{2,\infty}(\Omega;\mathbb{R}))]^n, \quad \vec{r}_\Gamma \in [L^\infty(0,T;W^{2,\infty}(\Gamma_1; \mathbb{R}))]^n$$ 
 with $\vec{r}=\vec{r}_\Gamma$ on $\Gamma_1\times (0,T)$, $\vec{r}\cdot \nu \leq 0$ on $\partial \Omega\times (0,T)$. We also consider that 
\begin{align}
\label{hip:02}
(q_0,q_{\Gamma,0})\in L^\infty(0,T;W^{1,\infty}(\Omega)\times W^{1,\infty}(\Gamma_1)).
\end{align}
We point out that the choice of $(\vec{q}_1,\vec{q}_{\Gamma,1})$, $(q_0,q_{\Gamma,0})$ and its assumptions comes from the wellposedness of system  \eqref{intro:problem:01} and also of its adjoint, which requires more space regularity. For more details about this subject, see Section \ref{section:existence:uniqueness}, problem \eqref{adjoint:z} and \eqref{def:q:qg:adjoint}.

Then, thanks to Theorem \ref{Thm:Carleman} we obtain

\begin{theorem}
	\label{Thm:Controllability}
	Assume the hypotheses of Theorem \ref{Thm:Carleman}. In addition, assume \eqref{hip:01} and \eqref{hip:02}.	
	Then, for all initial states $(y_0,y_{\Gamma,0}), (y_T,y_{\Gamma,T})\in \mathcal{V}'$, there exists a control $h\in L^2(\RM{\Gamma_*}\times (0,T))$ such that the associated solution 
	$(y, y_\Gamma)$ of \eqref{intro:problem:01} (defined in the sense of transposition) satisfies 
	\begin{align*}
	(y(T),y_{\Gamma}(T))=(y_T,y_{\Gamma,T})\text{ in }\Omega\times \Gamma_1.
	\end{align*} 
\end{theorem}

On the other hand, as a consequence of Corollary \ref{Corollary:Carleman} we can obtain a controllability result where a (distributed) control is acting on a subdomain of $\Omega$. In order to formulate this result, let $(y,y_\Gamma)$ be a solution of

\begin{equation}
\label{problem:control:distributed}
\begin{cases}
i\pt y + d\Delta y -\vec{q}_1 \cdot \nabla y +q_0y=\mathbbm{1}_\omega h,&\text{ in }\Omega\times (0,T),\\
i\pt y_\Gamma -d\pnu y +\delta \Delta_\Gamma y_\Gamma -\vec{q}_{\Gamma,1}\cdot \nabla_\Gamma y_\Gamma +q_{\Gamma,0}y_\Gamma =0,&\text{ on }\Gamma_1\times (0,T),\\
y=y_\Gamma,& \text{ on }\Gamma_1\times (0,T),\\
y=0,&\text{ on }\Gamma_0\times (0,T),\\
(y,y_\Gamma)(0)=(y_0,y_{\Gamma,0}),&\text{ in }\Omega\times \Gamma_1,
\end{cases}
\end{equation}

with $h\in L^2(\omega\times (0,T))$, with $\omega \subset \Omega$. Then, we have the following result:
\begin{theorem}\label{thm:controllability:distributed}
	Assume the  hypotheses of Corollary \ref{Corollary:Carleman}. Besides, we consider \eqref{hip:01} and \eqref{hip:02}.
	Then, for all $T>0$ and $(y_0,y_{\Gamma,0}),(y_T,y_{\Gamma,T})\in \mathcal{V}'$, there exists a control $h\in L^2(0,T;(H^1_{\Gamma_0}(\omega))')$ 
	such that the associated solution 
	$(y, y_\Gamma)$ of 
	\eqref{problem:control:distributed} 
	(defined in the sense of transposition) satisfies
	\begin{align*}
	(y(T),y_\Gamma(T))=(y_T,y_{\Gamma,T}),\text{ in }\Omega\times \Gamma_1.
	\end{align*}
\end{theorem}

The proof of Theorem \ref{thm:controllability:distributed} follows the same spirit of the proof of Theorem \ref{Thm:Controllability}. For this reason, in this paper we omit the proof of this result.

\begin{remark}
	We point out that some controllability results for conservative Schr\"odinger equation are obtained from the exact controllability for the wave equation, see for instance \cite{lebeau1992controle}, \cite{ervedoza2008observability} and generalized in \cite{miller2005controllability} for conservative systems by using a transmutation control technique. In case of \eqref{intro:problem:01} with 
	$(\vec{q}_1,\vec{q}_{\Gamma,1})=(0,0)$ and $(q_0,q_{\Gamma,0})=(0,0)$, a result can be obtained by applying these arguments 
	but with an additional control acting on a subset of $\Gamma_1$.
\end{remark}
The rest of the paper is organized as follows. In Section \ref{section:existence:uniqueness} we prove some results concerning existence and uniqueness of systems like \eqref{intro:problem:01} and \eqref{problem:control:distributed}. In Section \ref{section:proof:carleman} we prove the Carleman estimate obtained in Theorem \ref{Thm:Carleman}  and in Section \ref{section:proof:corollary:carleman} we prove the Corollary \ref{Corollary:Carleman} by using a suitable cut-off function. Finally, in Section \ref{section:proof:thm:controllability:boundary} we prove the Theorem \ref{Thm:Controllability}, which is equivalent to prove the so-called observability inequality associated to the adjoint system. 
\section{Existence and uniqueness of solutions}
\label{section:existence:uniqueness}

In this section, we provide existence and uniqueness results of solutions for Schr\"odinger equation with dynamic boundary conditions. We also give a hidden regularity property and define solutions in the sense of transposition for such systems.

Without loss of generality, it is sufficient to analyze the problem  \eqref{intro:problem:01} considering \eqref{hip:01} with $(\pi, \pi_\Gamma) = (0,0)$. 
That is to say, from now on, we will consider 
\begin{align}
	\label{simplified:problem:01}
	\begin{cases}
		i\pt y+d\Delta y +i\vec{r}\cdot \nabla y+q_0 y=0,&\text{ in }\Omega \times (0,T),\\
		i\pt y_\Gamma -d\pnu y+\delta \Delta_\Gamma y_\Gamma +i\vec{r}_\Gamma \cdot \nabla_\Gamma y_\Gamma + q_{\Gamma ,0}y_\Gamma =0,&\text{ on }\Gamma_1 \times (0,T),\\
		y=y_\Gamma,&\text{ on }\Gamma_1\times (0,T),\\
		y=\mathbbm{1}_{\Gamma_*}h,&\text{ on }\Gamma_0\times (0,T),\\
		(y(0),y_\Gamma (0))=(y_0,y_{\Gamma,0}),&\text{ in }\Omega\times \Gamma_1,
	\end{cases}
\end{align} 
with $\vec r$ and $\vec r_\Gamma$ real valued. 
Otherwise, we apply the change of variables $\tilde y(x,t)=e^{-\pi(x,t)/2}{y}(x,t)$ in $\Omega\times (0,T)$ and $\tilde y_\Gamma (x,t)=e^{-\pi(x,t)/2} {y}_\Gamma(x,t)$ on $\Gamma_1\times (0,T)$, where $(y,y_\Gamma)$ is a solution of \eqref{intro:problem:01}. Then, arguing as \cite[Appendix A]{Lasiecka01}, the new variable $(\tilde{y},\tilde{y}_\Gamma)$ satisfies a problem of the form \eqref{simplified:problem:01}. 

\subsection{Existence and regularity of solutions}
Given $d,\delta>0$, we consider the problem 
\begin{align}
\label{WP:problem:01}
	\begin{cases}
		\pt u-di\Delta u +\vec{\rho}_1\cdot \nabla u +\rho_0 u=g,&\text{ in } \Omega\times (0,T),\\
		\pt u +di\pnu u-\delta i\Delta_\Gamma u_\Gamma +\vec{\rho}_{\Gamma,1}\cdot \nabla_\Gamma u_\Gamma +\rho_{\Gamma,0}u_\Gamma=g_\Gamma,&\text{ on }\Gamma_1\times (0,T),\\
		u=u_\Gamma,&\text{ on }\Gamma_1\times (0,T),\\
		u=0,&\text{ on }\Gamma_0 \times (0,t),\\
		(u(0),u_\Gamma (0))=(u_0,u_{\Gamma,0}),&\text{ in }\Omega\times \Gamma_1,
	\end{cases}
\end{align}
where $(\vec{\rho}_1,\vec{\rho}_{\Gamma,1})$ is real valued and $(\rho_0,\rho_{\Gamma,0}), (g,g_\Gamma)$ are complex-valued. The next results establish $L^2$ and $H^1$ estimates of the weak solutions of the problem \eqref{WP:problem:01}, respectively. The proofs, as usual,  are based on energy estimates and 
density arguments. 
\begin{proposition}
\label{prop:WP:01}
	Suppose that $(u_0,u_{\Gamma,0})\in \mathcal{H}$, 
	$\vec{\rho}_1\in [L^\infty(0,T;W^{1,\infty}(\Omega;\mathbb{R}))]^n$. Moreover, assume that $\vec{\rho}_{\Gamma,1}\in 
	[L^\infty(0,T;W^{1,\infty}(\Gamma_1;\mathbb{R}))]^n$, with $\vec{\rho}_{1}=\vec{\rho}_{\Gamma,1}$ on $\Gamma_1\times (0,T)$, $(\rho_0,\rho_{\Gamma,0})\in L^\infty(0,T;L^\infty(\Omega\times \Gamma_1))$, and $(g,g_\Gamma)\in L^1(0,T;\mathcal{H})$. Then, there exists a positive constant $C$ such that the weak solution $(u,u_\Gamma)\in C^0([0,T];\mathcal{H})$ of \eqref{WP:problem:01} satisfies 
	\begin{align}
	\label{prop:WP:01:conclu}
		\max_{t\in [0,T]}\|(u,u_\Gamma)(t)\|_{\mathcal{H}} ^2\leq C \left(\|(g,g_\Gamma)\|_{L^1(0,T;\mathcal{H})} ^2+ \|(u_0,u_{\Gamma,0})\|_{\mathcal{H}}^2 \right). 
	\end{align}  
\end{proposition}
\begin{proof}
	Firstly, we multiply the first equation of \eqref{WP:problem:01} by $\overline{u}$ and integrate in $\Omega$. Secondly, we multiply the second equation by $\overline{u}_\Gamma$ and integrate on $\Gamma_1 \times (0,T)$. Next, we add these identities and take the real part on the obtained equation. This yields
	\begin{align}
		\label{WP:01:eq:01}
		\begin{split} 
		&\Re  \int_\Omega \overline{u} \pt u dx +d\Im \int_\Omega \overline{u}\Delta u dx + \Re \int_\Omega \overline{u}(\vec{\rho}_1 \cdot \nabla u)dx +\Re \int_\Omega \rho_0 |u|^2 dx\\
		&+\Re \int_{\Gamma_1} \overline{u}_\Gamma \pt u_\Gamma d\sigma -d\Im \int_{\Gamma_1} \overline{u}_\Gamma \pnu u d\sigma +d\Im \int_{\Gamma_1}  \overline{u}_\Gamma  \Delta_\Gamma u_\Gamma d\sigma \\
		&+\Re \int_{\Gamma_1} \overline{u}_\Gamma (\vec{\rho}_{\Gamma,1}\cdot \nabla_\Gamma u_\Gamma) d\sigma +\Re \int_{\Gamma_1} \rho_{\Gamma,0} |u_\Gamma|^2 d\sigma = \Re \int_\Omega \overline{u}g dx + \Re \int_{\Gamma_1} \overline{u}_\Gamma g_\Gamma d\sigma,
		\end{split}
	\end{align} 
	a.e. in $(0,T)$. We notice that,
	\begin{align}
	\label{WP:01:eq:02}
		\Re \int_\Omega \overline{u} \pt u dx +\Re \int_{\Gamma_1} \overline{u}_\Gamma \pt u_\Gamma d\sigma =\dfrac{d}{dt} \left(\dfrac{1}{2} \int_\Omega |u|^2 dx +\dfrac{1}{2} \int_{\Gamma_1} |u_\Gamma|^2 d\sigma \right). 
	\end{align}
	Moreover, integration by parts and surface divergence theorem implies that
	\begin{align}
	\label{WP:01:eq:03}
		d\Im \int_\Omega \overline{u}\Delta u dx -d\Im \int_{\Gamma_1} \overline{u}_\Gamma \pnu u d\sigma +\delta \Im \int_{\Gamma_1} \overline{u}_\Gamma \Delta_\Gamma u_\Gamma d\sigma=0, 
	\end{align}
	where we have used that $u=u_\Gamma$ on $\Gamma_1\times (0,T)$. In addition, 
	\begin{align}
	\label{WP:01:eq:04}
	\begin{split} 
		&\Re \int_\Omega \overline{u} (\vec{\rho}_1 \cdot \nabla u) dx +\Re \int_{\Gamma_1} \overline{u}_\Gamma (\vec{\rho}_{\Gamma,1} \cdot \nabla_\Gamma u_\Gamma) d\sigma\\
		=&-\dfrac{1}{2} \int_\Omega \text{div}(\vec{\rho}_1) |u|^2 dx -\dfrac{1}{2} \int_{\Gamma_1} \text{div}_\Gamma (\vec{\rho}_{\Gamma,1}) |u_\Gamma|^2 d\sigma +\dfrac{1}{2} \int_{\Gamma_1} (\vec{\rho}_1 \cdot \nu) |u|^2 d\sigma.
	\end{split}  
	\end{align}
	Substituting \eqref{WP:01:eq:02}, \eqref{WP:01:eq:03} and \eqref{WP:01:eq:04} into \eqref{WP:01:eq:01} and applying Cauchy-Scharwz inequality we deduce that
	\begin{align*}
		\dfrac{d}{dt} \|(u,u_\Gamma)(t)\|_{\mathcal{H}}^2 \leq C \|(u,u_\Gamma)(t)\|_\mathcal{H}^2 + C \|(u,u_\Gamma)(t)\|_\mathcal{H}\|(g,g_\Gamma)(t)\|_{\mathcal{H}}. 
	\end{align*}
	Then, by Gronwall's and Holder's inequalities,
	we obtain \eqref{prop:WP:01:conclu}. This ends the proof of the Proposition \ref{prop:WP:01}.
\end{proof}
Under additional assumptions on $(\vec{\rho}_1,\vec{\rho}_{\Gamma,1})$ and $(\rho_0,\rho_{\Gamma,0})$, we get the following result. 
\begin{proposition}
\label{WP:prop:02}
	Suppose that $(u_0,u_{\Gamma,0})\in \mathcal{V}$ and $(g,g_\Gamma)\in L^1(0,T;\mathcal{V})$. Moreover, assume that
	\begin{align}
	\label{WP:prop:02:additional:assump}
	\vec{\rho}_1 \in [L^\infty (0,T; W^{1,\infty}(\Omega;\mathbb{R}))]^n, \quad   \vec{\rho}_{\Gamma,1} \in  [L^\infty (0,T; W^{1,\infty}(\Gamma_1;\mathbb{R}))]^n,
	\end{align}
	with $\vec{\rho}_1 =\vec{\rho}_{\Gamma,1},\text{ on }\Gamma_1\times (0,T)$ and $\vec{\rho}_1 \cdot \nu \leq 0$ on $\partial \Omega\times (0,T)$. We also assume that
	\begin{align*}
		\rho_0 \in L^\infty(0,T;W^{1,\infty}(\Omega)),\quad \rho_{\Gamma,0}\in L^\infty(0,T;W^{1,\infty}(\Gamma_1)),
	\end{align*}
	with $\rho_0=\rho_{\Gamma,0}$ on $\Gamma_1\times (0,T)$.
	 Then, the weak solution of \eqref{WP:problem:01} belongs to $C^0([0,T];\mathcal{V})$. Moreover, there exists $C>0$ such that
	\begin{align}
	\label{WP:prop:02:conclu}
		\max_{t\in [0,T]}\|(u,u_\Gamma)\|_{\mathcal{V}}^2 \leq C\left( \|(g,g_\Gamma)\|_{L^1(0,T;\mathcal{V})}^2 + \|(u_0,u_{\Gamma,0})\|_{\mathcal{V}}^2 \right). 
	\end{align}
\end{proposition}
\begin{proof}
	We multiply the first equation of \eqref{WP:problem:01} by $-d\Delta \overline{u}$ and integrate in $\Omega \times (0,T)$. Next, we multiply the second equation of \eqref{WP:problem:01} by $(-\delta \Delta_\Gamma \overline{u}_\Gamma + d\pnu \overline{u})$ and integrate on $\Gamma_1 \times (0,T)$. Then, adding these identities and taking the real part 
	we have
	\begin{align}
	\label{WP:problem:02:01}
		\begin{split}
			&-d\Re \int_\Omega \pt u \Delta \overline{u}dx -d\Re \int_\Omega (\vec{\rho}_1 \cdot\nabla u)\Delta \overline{u}dx -d\Re \int_\Omega pu\Delta \overline{u} dx\\
			&+\Re \int_{\Gamma_1} \pt u_\Gamma (-\delta \Delta_\Gamma \overline{u}_\Gamma +d\pnu \overline{u})d\sigma +\Re \int_{\Gamma_1} (\vec{\rho}_1 \cdot \nabla_\Gamma u_\Gamma)(-\delta \Delta_\Gamma \overline{u}_\Gamma +d\pnu \overline{u})d\sigma \\
			&+\Re \int_{\Gamma_1} p_\Gamma u_\Gamma (-\delta \Delta_\Gamma \overline{u}_\Gamma +d\pnu \overline{u})d\sigma\\
			=&-d\Re \int_\Omega g\Delta \overline{u} dx + \Re \int_{\Gamma_1} g_\Gamma (-\delta \Delta_\Gamma \overline{u}_\Gamma +d\pnu \overline{u})d\sigma. 
		\end{split}
	\end{align}
	Integration by parts and surface divergence theorem imply that
	\begin{align}
	\label{WP:problem:02:02}
		\begin{split}
			&-d\Re \int_\Omega \pt u \Delta \overline{u} dx + \Re \int_{\Gamma_1} \pt u_\Gamma (-\delta \Delta _\Gamma \overline{u}_\Gamma +d \pnu \overline{u})d\sigma\\
			=& \dfrac{d}{dt} \left(\dfrac{1}{2}d \int_\Omega |\nabla u|^2 dx + \dfrac{1}{2}\delta \int_{\Gamma_1} |\nabla_\Gamma u_\Gamma|^2 d\sigma \right). 
		\end{split}
	\end{align}   
	On the other hand,
	\begin{align*}
		\begin{split} 
		\Lambda:=&-d\Re \int_\Omega (\vec{\rho}_1 \cdot \nabla u) \Delta \overline{u} dx + \Re \int_{\Gamma_1} (\vec{\rho}_{\Gamma,1} \cdot \nabla_\Gamma u_\Gamma) (-\delta \Delta_\Gamma \overline{u}_\Gamma +d\pnu \overline{u})d\sigma \\
		=&d\Re \int_\Omega \nabla (\vec{\rho}_1 \cdot \nabla u)\cdot \nabla \overline{u} dx +\dfrac{1}{2} d\int_{\partial \Omega} (\vec{\rho}_1 \cdot \nabla u) \pnu \overline{u} d\sigma \\
		&+\delta \Re \int_{\Gamma_1} \nabla_\Gamma (\vec{\rho}_{\Gamma,1} \cdot \nabla_\Gamma u_\Gamma)\cdot \nabla_\Gamma \overline{u}_\Gamma d\sigma + d\Re \int_{\Gamma_1} (\vec{\rho}_{\Gamma,1} \cdot \nabla_\Gamma u_\Gamma) \pnu \overline{u}d\sigma.   
		\end{split}  
	\end{align*}
	Since
	\begin{align*} 
	\nabla (\vec{\rho}_1 \cdot \nabla u)\cdot \nabla \overline{u}&=\vec{\nabla} \rho_{1}(\nabla u,\nabla \overline{u}) +\dfrac{1}{2} \vec{\rho}_1 \cdot \nabla (|\nabla u|^2),\\
	\nabla_\Gamma (\vec{\rho}_{\Gamma,1} \cdot \nabla_\Gamma u_\Gamma) \cdot \nabla_\Gamma \overline{u}_\Gamma &= \vec{\nabla}_\Gamma \vec{\rho}_{\Gamma,1} (\nabla_\Gamma u_\Gamma, \nabla_\Gamma \overline{u}_\Gamma)+\dfrac{1}{2} \vec{\rho}_{\Gamma,1} \cdot \nabla_\Gamma (|\nabla_\Gamma u_\Gamma|^2),	\end{align*}
	and that $ \nabla u =\nabla_\Gamma u + (\pnu u) \nu$ on $\Gamma_1\times (0,T)$, we obtain
	\begin{align}
		\label{WP:problem:02:03}
		\begin{split} 
		\Lambda=&d\Re \int_\Omega \vec{\nabla}\vec{\rho}_1 (\nabla u,\nabla \overline{u}) dx +\delta \Re \int_{\Gamma_1} \vec{\nabla}\vec{\rho}_{\Gamma,1} (\nabla_\Gamma u_\Gamma, \nabla_\Gamma \overline{u}_\Gamma) d\sigma\\
		&-\dfrac{1}{2} d\int_\Omega \text{div}(\vec{\rho}_1) |\nabla u|^2 -\dfrac{1}{2}\delta \int_{\Gamma_1} \text{div} (\vec{\rho}_{\Gamma,1}) |\nabla_\Gamma u_\Gamma|^2 d\sigma\\
		&-\dfrac{1}{2}d \int_{\partial \Omega} (\vec{\rho}_1 \cdot \nu )|\pnu u|^2 d\sigma +\dfrac{1}{2} d \int_{\Gamma_1} (\vec{\rho}_1 \cdot \nu )|\nabla_\Gamma u_\Gamma|^2 d\sigma,
		\end{split}  
	\end{align}	   
	where we have integrated by parts and used the fact that $\vec{\rho}_{1}=\vec{\rho}_{\Gamma,1}$ on $\Gamma_1\times (0,T)$. In addition, since $\rho_0=\rho_{\Gamma,0}$ on $\Gamma_1\times (0,T)$, we have
	\begin{align}
		\label{WP:problem:02:04}
		\begin{split} 
		&-d\Re \int_\Omega pu \Delta \overline{u} dx +\Re \int_{\Gamma_1} p_\Gamma u_\Gamma (-\delta \Delta_\Gamma \overline{u}_\Gamma +d\pnu \overline{u}) d\sigma\\
		=& d\int_\Omega \Re (p) |\nabla u|^2 + \delta \int_{\Gamma_1} \Re (p_\Gamma) |\nabla_\Gamma u_\Gamma|^2 d\sigma \\
		&+d\Re \int_\Omega u\nabla p \cdot \nabla \overline{u} dx +\delta \Re \int_{\Gamma_1} u_\Gamma \nabla_\Gamma p_\Gamma \cdot \nabla_\Gamma \overline{u}_\Gamma d\sigma.  
		\end{split} 
	\end{align}
	Now, since $g=0$ on $\Gamma_0 \times (0,T)$ and $g=g_\Gamma$ on $\Gamma_1\times (0,T)$, we deduce that
	\begin{align}
	\label{WP:problem:02:05}
		\begin{split}
			&-d\Re \int_\Omega g\Delta \overline{u} dx +\Re \int_{\Gamma_1} g_\Gamma (-\delta \Delta_\Gamma \overline{u}_\Gamma +d\pnu \overline{u}) d\sigma \\
			=& d\Re \int_\Omega \nabla g\cdot \nabla \overline{u} dx + \delta \Re \int_{\Gamma_1} \nabla_\Gamma g \cdot \nabla_\Gamma \overline{u}_\Gamma d\sigma. 
		\end{split}
	\end{align} 
	Then, substituting \eqref{WP:problem:02:02}, \eqref{WP:problem:02:03}, \eqref{WP:problem:02:04}, \eqref{WP:problem:02:05} into \eqref{WP:problem:02:01} and using the fact that $\vec{\rho}_1 \cdot \nu \leq 0$ on $\partial \Omega \times (0,T)$, we can assert that
	\begin{align*}
		\dfrac{d}{dt} \|(u,u_\Gamma)(t)\|_{\mathcal{V}}^2 \leq C\|(u,u_\Gamma)(t)\|_{\mathcal{V}}^2 
+		C\|(u,u_\Gamma)(t)\|_{\mathcal{V}}\|(g,g_\Gamma)(t)\|_\mathcal{V} 
	\end{align*}
	Then, applying Gronwall's and H\"older's inequalities we easily deduce \eqref{WP:prop:02:conclu} and the proof of the Proposition \ref{WP:prop:02} is finished.  
\end{proof}
\begin{remark}
 	It is also possible to obtain \eqref{WP:prop:02:conclu} considering regularity assumptions on $(\vec{\rho}_1,\vec{\rho}_{\Gamma,1})$ and $(\rho_0,\rho_{\Gamma,0})$ in the time variable. However, in order to simplify the presentation of the rest of the results in this paper, we do not consider these situations.  
\end{remark}

\subsection{Hidden regularity}

In this section, we devote to deduce a hidden regularity result for solutions of the Schr\"odinger equation with dynamic boundary conditions. For this purposes, we start giving the following identity. 

\begin{proposition}  \label{identq}
	Suppose that $\vec{q} \in C^2(\overline{\Omega}\times [0,T];\mathbb{R}^n)$. Moreover, consider $\vec{\rho}_1 \in [L^{\infty}(0,T; W^{1,\infty}(\Omega;\mathbb{R}))]^n$, $ 
	\vec{\rho}_{\Gamma,1} \in [ L^{\infty}(0,T;  W^{1,\infty}(\Gamma_1;\mathbb{R}))]^n$ and $\rho_0 \in L^\infty(\Omega \times (0,T))$, $\rho_{\Gamma,0} \in  L^\infty(\Gamma_1\times (0,T))$.
	Then, for each $d,\delta>0$, $(u_0,u_{\Gamma,0})\in \mathcal{V}$ and  $(g,g_\Gamma) \in L^1(0,T; \mathcal V)$, the solution $(u,u_\Gamma)$ of \eqref{WP:problem:01} satisfies
	\begin{align}
		\label{eq:prop:identity}
		&\dfrac{1}{2}d\int_0^T\int_{\partial \Omega} (\vec{q}\cdot \nu) |\pnu u|^2 dxdt=X+Y+F,
	\end{align}
	where $X$, $Y$ and $F$ are given by
	\begin{align*}
		\begin{split} 
		X=&-\dfrac{1}{2}\delta \IGT{(\vec{q}\cdot \nu) |\nabla_\Gamma u_\Gamma|^2} -\dfrac{1}{2} \Im \IGT{(\vec{q}\cdot \nu)\rho_{\Gamma,0} |u_\Gamma|^2}\\
		&+\dfrac{1}{2} d\IOT{\vec{\nabla} \vec{q}(\nabla u,\nabla \overline{u})} +\dfrac{1}{2}d\IOT{\text{div}(\vec{q}) |\nabla u|^2}\\
		&+\dfrac{1}{2}\delta \int_0^T\int_{\Gamma_1} \vec{\nabla}_\Gamma \vec{q}(\nabla_\Gamma u_\Gamma, \nabla_\Gamma \overline{u}_\Gamma) d\sigma dt +\dfrac{1}{2}\delta \IGT{\text{div}(\vec{q}) |\nabla_\Gamma u_\Gamma|^2}\\
		&+\Im \IOT{(\vec{\rho}_1\cdot \nabla u)(\vec{q}\cdot \nabla \vec{u})}+\dfrac{1}{2}\Im \IOT{\rho_0 \text{div}(\vec{q}) |u|^2}\\
		&+\Im \IGT{(\vec{\rho}_{\Gamma,1} \cdot \nabla_\Gamma u_\Gamma) (\vec{q}\cdot \nabla_\Gamma u_\Gamma)} +\dfrac{1}{2}\Im \IGT{\rho_{\Gamma,0} \text{div}(\vec{q}) |u_\Gamma|^2},
		\end{split}
	\end{align*}
	\begin{align*}
		Y=&\dfrac{1}{2}\Im \IOT{\overline{u} \nabla u \cdot \pt \vec{q}} -\dfrac{1}{2} \Im \int_\Omega \overline{u} \nabla u \cdot \vec{q}\bigg|_0^T dx\\
		& -\dfrac{1}{2} d\Re \IGT{(\vec{q}\cdot \nu )\overline{u}_\Gamma \pnu u} - \dfrac{1}{2}\Im \IGT{(\vec{q}\cdot \nabla_\Gamma \nu) (\vec{\rho}_{\Gamma,1} \cdot u_\Gamma)\overline{u}_\Gamma} \\
		&+\dfrac{1}{2}\Im \IGT{(\nabla_\Gamma u_\Gamma \cdot \pt \vec{q})\overline{u}_\Gamma}-\dfrac{1}{2}\Im \int_{\Gamma_1} (\vec{q}\cdot \nabla_\Gamma u_\Gamma)\overline{u}_\Gamma \bigg|_0^T d\sigma\\
		&+\dfrac{1}{2}d\Re \IOT{\overline{u}\nabla (\text{div}(\vec{q}))\cdot \nabla u} +\dfrac{1}{2} d\Re \IGT{(\pnu u) \vec{q}\cdot \nabla_\Gamma u_\Gamma}\\
		&+\dfrac{1}{2}\delta \Re \IGT{\nabla_\Gamma (\text{div}(\vec{q}) \cdot \nabla_\Gamma u_\Gamma)\overline{u}_\Gamma} +\dfrac{1}{2}\Im \IOT{\text{div}(\vec{q}) (\vec{\rho}\cdot \nabla u)\overline{u}}\\
		&+\Im \IOT{\rho_0 (\vec{q}\cdot \nabla \overline{u})u}+\dfrac{1}{2}\IGT{\text{div} (\vec{q}) (\vec{\rho}_{\Gamma,1} \cdot \nabla_\Gamma u_\Gamma)\overline{u}_\Gamma}\\
		&+\Im \IGT{\rho_{\Gamma,0} (\vec{q}\cdot \nabla_\Gamma \overline{u}_\Gamma) u_\Gamma},       
	\end{align*}
	and 
	\begin{align*}
		F=&\dfrac{1}{2}\Im \IGT{(\vec{q}\cdot \nu ) \overline{u}_\Gamma g_\Gamma} - \Im \IOT{g\left(\vec{q}\cdot \nabla \overline{u} + \dfrac{1}{2} \text{div}(\vec{q}) \overline{u}  \right) }\\
		&-\Im \IGT{g_\Gamma \left(\vec{q}\cdot \nabla_\Gamma u_\Gamma +\dfrac{1}{2} \text{div}(\vec{q}) \overline{u}_\Gamma \right)}.
	\end{align*}
\end{proposition}
\begin{proof} 
The proof of the Proposition \ref{identq} relies on the use of classical ideas concerning multiplier techniques.
 That is to say, we multiply the first equation of 
 \eqref{WP:problem:01} by $-\left(\vec{q}\cdot\nabla \overline{u} + \dfrac{1}{2} \text{div}(\vec{q})\overline{u}\right)$ and
 we integrate  in $\Omega\times (0,T)$; 
 similarly, we multiply the second one by $-\left(\vec{q}\cdot\nabla_\Gamma \overline{u}_\Gamma + \dfrac{1}{2} \text{div}(\vec{q})\overline{u}_\Gamma\right)$
and the result is integrated  on $\Gamma_1 \times (0,T)$. 
 Then we add the above identities and take the imaginary part. Contrary to the case of the homogeneous Dirichlet boundary conditions (see for instance \cite{machtyngier1994exact}), in our case several new boundary terms defined on $\Gamma_1\times (0,T)$ appear. To deal with them, we use the surface divergence theorem and the second equation of \eqref{WP:problem:01} to estimate $\pt u_\Gamma$ on $\Gamma_1\times (0,T)$. The rest of the proof runs as in \cite{lions1988controlabilite} and \cite{machtyngier1994exact} for wave and Schr\"odinger equation with Dirichlet boundary conditions, respectively. We omit the details.
\end{proof}

With this estimate at hand, we have the next result of hidden regularity for the Schr\"odinger equation with dynamic boundary conditions: 
\begin{proposition} \label{regesp}
	Consider the same assumptions of Proposition \ref{WP:prop:02}. Then, the derivative $\pnu u$ of solution of \eqref{WP:problem:01} belongs to $L^2(\partial\Omega\times (0,T))$. Moreover, there exists a positive constant $C$ such that
	\begin{equation} \label{RegNorm}
	\int_0^T \int_{\partial \Omega}  |\pnu u|^2 d\sigma dt 
	\leq 
	C\left( \| (g, g_\Gamma) \|_{L^1(0,T; \mathcal V)}^2 + \|(u_0, u_{\Gamma, 0}) \|_{\mathcal V}^2 
	\right).
	\end{equation} 
\end{proposition}
\begin{proof}
	We start choosing a smooth vector field $\vec{q} \in C^2(\overline \Omega, \R^n)$ such that $q = \nu$ on $\partial \Omega$. See, for instance, \cite{lions1988controlabilite} for the proof of the existence of such a vector field. Then, by Proposition \ref{identq}, equation \eqref{eq:prop:identity} reads
	\begin{align}
	\label{prop:hidden:01}
		\dfrac{1}{2}d \|\pnu u\|_{L^2(0,T;L^2(\partial \Omega))}^2 = X+Y+F.
	\end{align}
	Since $(p,p_\Gamma)$ and $\vec{q}$ and its derivatives are bounded in $L^\infty$, we have that
	\begin{align}
		\nonumber 
		|X|\leq &C\|(u,u_\Gamma)\|_{C^0([0,T]; \mathcal{V})}\\
		\label{prop:hidden:02}
		\leq C &\left(\|(u_0,u_{\Gamma,0})\|_{\mathcal{V}}^2 + C\|(g,g_\Gamma)\|_{L^1(0,T;\mathcal{V})}^2 \right), 
	\end{align}
	where we have applied the Proposition \ref{WP:prop:02} in the last inequality. In the same manner, by Young's inequality, $Y$ can be estimated as
	\begin{align}
	\label{prop:hidden:03}
		|Y|\leq C \|(u_0,u_{\Gamma,0})\|_{\mathcal{V}}^2 + \|(g,g_\Gamma)\|_{L^1(0,T;\mathcal{V})}^2 + \dfrac{1}{4}d \|\pnu u\|_{L^2(0,T;L^2(\partial \Omega))}^2.
	\end{align}
	Moreover, by Cauchy-Scharwz inequality we can assert that
	\begin{align}
	\label{prop:hidden:04}
		|F|\leq C \|(u_0,u_{\Gamma,0})\|_{\mathcal{V}}^2 + C\|(g,g_\Gamma)\|_{L^1(0,T;\mathcal{V})}^2.
	\end{align}
	Substituting \eqref{prop:hidden:02}, \eqref{prop:hidden:03} and \eqref{prop:hidden:04} into \eqref{prop:hidden:01}, we obtain \eqref{RegNorm}. This finishes the proof of the Proposition \ref{regesp}.
\end{proof}

\subsection{Solutions by transposition}

To define solutions of the Schr\"odinger equation with dynamic boundary condition and boundary control \eqref{simplified:problem:01}, we consider  the auxiliary system given by
\begin{align}
\label{adj1} 
\begin{cases}
i\pt \phi +d\Delta \phi +i\vec{r}\cdot \nabla \phi +q\phi = f,&\text{ in }\Omega\times (0,T),\\
i\pt \phi_\Gamma -d\pnu \phi  +\delta \Delta_\Gamma \phi_\Gamma +i\vec{r}_{\Gamma}\cdot \nabla_\Gamma \phi_\Gamma +q_\Gamma \phi_\Gamma = f_\Gamma,&\text{ on }\Gamma_1\times (0,T),\\
\phi=\phi_\Gamma,&\text{ on }\Gamma_1\times (0,T),\\
\phi=0,&\text{ on }\Gamma_0\times (0,T),\\
(\phi,\phi_\Gamma)(T, \cdot)=(\phi_T, \phi_{\Gamma,T}),&\text{ in }\Omega\times \Gamma_1,
\end{cases}
\end{align}
where the pair $(q,q_\Gamma)$ is given by
\begin{align}
	q=i\text{div}(\vec{r}) +\overline{q}_0,\quad q_\Gamma =i\text{div}_\Gamma (\vec{r}_\Gamma)-i\vec{r}\cdot \nu +\overline{q}_{\Gamma,0}.
\end{align}

According to Proposition \ref{WP:prop:02} we have that, for 
each $(\phi_T, \phi_{\Gamma,T}) \in {\mathcal V} $ and  $(f, f_\Gamma)  \in L^1(0,T; {\mathcal V})$, system \eqref{adj1} posses a  unique solution 
\begin{equation} \label{condin}
(\phi, \phi_\Gamma)  \in C([0,T]; {\mathcal V}).
\end{equation}
We can now define a solution by transposition.
\begin{definition}
	Given $(y_0,{y_\Gamma}_0)\in \mathcal{V}'$ and
	a control $h \in  {L^2(\Gamma_*\times (0,T))}$, we will say that  {the pair} $(y, y_\Gamma) \in C([0,T]; {\mathcal V}')$ is a solution of system \eqref{intro:problem:01} in the sense of transposition if the following identity holds
	\begin{equation} \label{transp}
	\int_0^T \left< (y,y_\Gamma),  (f, f_\Gamma) \right>_{\mathcal{V}',\mathcal{V}}dt  
	=
	d\int_0^T \int_{\Gamma_0} \mathbbm{1}_{\Gamma_*}h  \overline {\pnu \phi} d\sigma dt
	+i  \left< (y_0,{y_\Gamma}_0),  (\phi, \phi_\Gamma)(0, \cdot) \right>_{\mathcal{V}',\mathcal{V}}
	\end{equation}
	for each $(f, f_\Gamma)  \in L^1(0,T; {\mathcal V})$,
	where $(\phi, \phi_\Gamma)$   is the solution of \eqref{adj1} with  $(\phi_T, \phi_{\Gamma, T}) = (0,0)$.
\end{definition}
\begin{theorem}
	For each $(y_0, y_{\Gamma,0}) \in {\mathcal V}'$ and $ {h \in L^2(\Gamma_* \times (0,T))}$, system \eqref{intro:problem:01} has a unique solution
	$(y,y_\Gamma) \in C([0,T]; {\mathcal V}')$. Moreover, there exists $C>0$ such that
	\begin{equation*}
	\|(y,y_\Gamma) \|_{L^\infty(0,T; \mathcal{V}')} \leq C \left( \|(y_0, y_{\Gamma,0}) \|_{\mathcal{V}'} + \|h\|_{ {L^2(\Gamma_*\times (0,T))}}  \right).
	\end{equation*}
\end{theorem}
\begin{proof}
	Given
	$(y_0,{y_\Gamma}_0)\in \mathcal{V}'$ and  ${h \in L^2( \Gamma_* \times (0,T) ) }$, 
	from Proposition \ref{regesp} and \eqref{condin}, the linear functional defined by
	\begin{equation*}
	(f,f_\Gamma) \in L^1(0,T; \mathcal V) \mapsto d 
	\int_0^T \int_{\Gamma_0} \mathbbm{1}_{\Gamma_*} h  \overline {\pnu \phi} d\sigma dt
	+i  \left< (y_0,{y_\Gamma}_0),  (\phi(0, \cdot), \phi_\Gamma(0, \cdot)) \right>_{\mathcal{V}',\mathcal{V}}
	\end{equation*}
	is well defined and continuous. By duality, there exists 
	$(y,y_\Gamma) \in L^\infty(0,T; \mathcal V')$ such that we have \eqref{transp}. 
	By classical density arguments we get that $(y,y_\Gamma) \in C([0,T]; \mathcal{V}')$. 
\end{proof}

Now we will consider the problem of exact controllability. By linearity, it is enough to study the case of 
null initial conditions. It is direct to see that $h$ is a control for system \eqref{intro:problem:01} driving $(y,y_\Gamma) =(0,0)$ 
to $(y(T, \cdot) ,{y_\Gamma}(T, \cdot)) \in \mathcal{V}'$ if
\begin{equation} \label{exactc}
d\int_0^T \int_{\Gamma_0} \mathbbm{1}_{\Gamma_*}h  \overline {\pnu \phi} d\sigma dt
=
i \left< (y(T, \cdot) ,{y_\Gamma}(T, \cdot)),  (\phi_T, \phi_{\Gamma,T}) \right>_{\mathcal{V}',\mathcal{V}}
\end{equation}
for each $ (\phi_T, \phi_{\Gamma, T})  \in {\mathcal V}$,
where $(\phi, \phi_\Gamma) $   is the solution of \eqref{adj1} with  $(f, f_\Gamma)  = (0,0)$. Therefore we get, by classical abstract results  {(see for instance \cite{tucsnak2009observation})}, the following characterization. 
\begin{proposition}
	\label{prop:exact:controllability}
	System \eqref{intro:problem:01} is exactly controllable in time $T>0$ if and only if there exists a constant $C>0$ such that
	\begin{equation} \label{ObsIneq}
	\| (\phi_T, \phi_{\Gamma,T}) \|_{\mathcal V}^2 \leq C \int_0^T\int_{\Gamma_*}{|\pnu \phi|^2} d\sigma dt,
	\end{equation}
	for all  $ (\phi_T, \phi_{\Gamma, T})  \in {\mathcal V}$.
\end{proposition}
\section{Proofs of  Theorem \ref{Thm:Carleman} and Corollary \ref{Corollary:Carleman}}
\label{section:proof:carleman}
In order to prove these results, we point out that it is sufficient to prove inequalities  \eqref{Carleman:estimate:01} and \eqref{Carleman:estimate:distributed} for functions $(v,v_\Gamma)\in C^\infty((\overline{\Omega}\times \Gamma_1 ) \times  [0,T])$ which satisfies $v=0$ on $\Gamma_0\times (0,T)$ and $v=v_\Gamma$ on $\Gamma_1\times (0,T)$ {and argue by density as usual}. Then, {due to the regularity of each function and the fact that $v=v_\Gamma$}, from now on we shall write $v$ instead of $v_\Gamma$ on $\partial \Omega\times (0,T)$.

\subsection{Proof  of  Theorem \ref{Thm:Carleman}}

Here we will  prove the Carleman estimate of Theorem \ref{Thm:Carleman}. In order to do that, we shall follow the classical conjugation for suitable operators involving the Schr\"odinger equation with dynamic boundary conditions. However, contrary to the case of Dirichlet boundary conditions, several additional  boundary terms arise from the dynamic boundary conditions of Wentzell type. To treat them, we have to repeat some steps from \cite{baudouin2008inverse} in a modified form.  

In order to simplify the presentation, the proof of Theorem \ref{Thm:Carleman} is done in several steps.

\noindent $\bullet$ {\bf Step 1: Setting.}

In this step, we define the operators to conjugate in $\Omega\times (0,T)$ and on $\Gamma_1\times (0,T)$. To do this, we write $v=e^{s\varphi}w$ and compute $Pw=e^{-s\varphi} L(e^{s\varphi}w)$ in $\Omega\times (0,T)$ and $Qw=e^{-s\varphi}N(e^{s\varphi}w)$ on $\Gamma_1\times (0,T)$. 

Firstly, we write the operator $P$ as follows
\begin{align}
\label{eq:P}
Pw=P_1 w+P_2 w+Rw,\quad \text{ in }\Omega\times (0,T),
\end{align}  
where the operators $P_1$, $P_2$ and $R$ are given by
\begin{align*}
P_1 w=ds^2 |\nabla \varphi|^2 w+d\Delta w +i\pt w,\quad P_2 w=ds\Delta \varphi w +2ds\nabla \varphi \cdot \nabla w +is\pt \varphi w,
\end{align*}
and 
\begin{align*}
Rw =\vec{q}_1\cdot \nabla w + s\nabla \varphi \cdot \vec{q}_1 w+ q_0w.
\end{align*}

Secondly, we write
\begin{align}
\label{eq:Q}
Qw=Q_1 w+Q_2 w+R_\Gamma w,\text{ on }\Gamma_1\times (0,T),
\end{align}
where $Q_1, Q_2$ and $R_\Gamma$ are given by
\begin{align*}
Q_1 w=\delta \Delta_\Gamma w +i\pt w,\quad Q_2  w=-ds \pnu \varphi w +is \pt \varphi w,
\end{align*}
and
\begin{align*} 
R_\Gamma w=-d\pnu w +\vec{q}_{\Gamma,1}\cdot \nabla_\Gamma w + s\nabla_\Gamma \varphi \cdot \vec{q}_{\Gamma,1} w+ q_{\Gamma,0} w.
\end{align*}

Now, taking the $L^2(\Omega\times (0,T))$-norm in \eqref{eq:P} and the $L^2(\Gamma_1 \times (0,T))$-norm in \eqref{eq:Q} we get
\begin{align}
\label{after:conjug}
\begin{split} 
&\IOT{|Pw-Rw|^2}+\IGT{|Qw-R_\Gamma w|^2}\\
=& \IOT{\left(|P_1 w|^2+ |P_2 w|^2 \right)}+ \IGT{\left( |Q_1 w|^2+ |Q_2 w|^2 \right) }\\
&+ 2 \Re \IOT{P_1 w\overline{P_2 w}}+2 \Re \IGT{Q_1 w \overline{Q_2 w}}.
\end{split}
\end{align}

The next two steps of the proof are devoted to compute the last two terms of the right-hand side of \eqref{after:conjug}.

\noindent $\bullet$ {\bf Step 2: Estimates in $\Omega\times (0,T)$.}

In this step, we compute the terms 
\begin{align*}
\Re \IOT{P_1 w \overline{P_2 w}}=\sum_{j=1}^3 \sum_{k=1}^3 I_{jk},
\end{align*}
where $I_{jk}$ denotes the $L^2$ real inner product of the $j^{\text{th}}-$term of $P_1w$ 
with the  $k^{\text{th}}$-term of $P_2w$. In order to do that, we shall take into account some properties of the function $\mu$ defined in \eqref{Mink}, 

which are stated in the following result. 
\begin{proposition}
Under the hypotheses of Theorem \ref{Thm:Carleman}, 
the function $\mu$ 
defined by \eqref{Mink} satisfies the following properties:
\begin{enumerate}
	\item $\mu\in C^4(\overline{\Omega})$.
	\item $\mu = 1$ on $\Gamma_1$. 
	\item $\nabla \mu \neq 0$ in $\overline{\Omega}$.
	\item There exists $c > 0$ such that $D^2 \mu (\xi, \xi) \geq c |\xi|^2$  in $\overline{\Omega}$ for all $\xi \in \mathbb{R}^n$.  
\end{enumerate}
\end{proposition}
\begin{proof}
We recall that  the origin belongs to the open set $\Omega_1$, and then,
there exists $\varepsilon>0$  such that 
$\overline{B_\varepsilon(0)} \cap \overline \Omega = \emptyset $. 
Taking into account hypothesis \eqref{regmu}, it is not difficult to see that, for each $x \in \Omega$, we have
$$ \mu(x) = \frac{\|x\|}{\left\| \Phi^{-1}\left(\frac{x}{\|x\|} \right)\right\|},
$$ 
and then, we 
obtain that $\mu$ is of class $C^4$ in $\overline \Omega$.
Moreover, directly from the definition of $\mu$ we get that $\Gamma_1 = \{ x \in \overline \Omega \,:\, \mu(x) =1\}$. 
Finally, assertions   3 and 4 
are proved in  \cite{baudouin2008inverse} (see also \cite{BMO}).
\end{proof}

In the following, we consider 
the following computations
\begin{align}
\label{estimate:der:weights:01}
\nabla \varphi=-\lambda \theta \nabla \psi,\quad \nabla \theta=\lambda \theta \nabla \psi, \quad \Delta \varphi =-\lambda^2\theta |\nabla \psi|^2 -\lambda \theta \Delta \psi,\text{ in }\Omega\times (0,T).
\end{align}

Moreover, 
\begin{align}
\label{estimate:der:weights:02}
\pnu \varphi=-\lambda \theta \pnu \psi ,\text{ on }\partial \Omega\times (0,T),\quad \text{ and } \nabla_\Gamma \theta=\nabla_\Gamma \varphi=0,\text{ on }\Gamma_1\times (0,T).
\end{align}

Here and subsequently, $C$ stands for a positive constant which is independent of the parameters $\lambda>0$ and $s>0$ and that might  change from line to line in our computations.

According to 
\eqref{estimate:der:weights:01}, the term $I_{11}$ can be written as

\begin{align}
\label{I11}  
\begin{split} 
I_{11}=&d^2 s^3 \IOT{\Delta \varphi |\nabla \varphi|^2 |w|^2}\\
=&-d^2s^3 \lambda^3 \IOT{(\lambda |\nabla \psi|^2 +\Delta \psi) |\nabla \psi|^2 \theta^3 |w|^2}.
\end{split} 
\end{align}

Also, we have
\begin{align} 
\nonumber
I_{12}=
& d^2 s^3  \IOT{|\nabla \varphi|^2  \nabla \varphi \cdot \nabla | w|^2 }\\
\nonumber
=&-d^2 s^3 \IOT{\nabla (|\nabla \varphi|^2)\cdot \nabla \varphi |w|^2}-d^2 s^3 \IOT{\Delta \varphi |\nabla \varphi|^2 |w|^2}\\
\nonumber
&+d^2 s^3 \IGT{|\nabla \varphi|^2 \pnu \varphi |w|^2} \\  
\nonumber
=&3d^2 s^3\lambda^4 \IOT{|\nabla \psi|^4 \theta^3 |w|^2} -d^2 s^3\lambda^3 \IGT{(\pnu \psi )^3 \theta^3 |w|^2}\\
&+d^2s^3\lambda^3 \IOT{(2\nabla^2 \psi(\nabla \psi,\nabla \psi)+\Delta \psi |\nabla \psi|^2)\theta^3 |w|^2},
\label{I12} 
\end{align}

%
where we have used that $w=0$ on $\Gamma_0\times (0,T)$. Now, the term $I_{13}$ is
\begin{align}
\label{I13} 
I_{13}= - ds^3 \Re i \IOT{\pt \varphi |\nabla \varphi|^2 |w|^2}=ds^3 \Im \IOT{\pt \varphi |\nabla \varphi|^2 |w|^2}=0.
\end{align}

Adding up the equations \eqref{I11}, \eqref{I12} and \eqref{I13} we get
\begin{align}
\nonumber 
\sum_{k=1}^3 I_{1k}=&2d^2 s^3 \lambda^4 \IOT{|\nabla \psi|^4 \theta^3|w|^2}\\
\nonumber 
&+2d^2 s^3\lambda^3 \IOT{\nabla^2 \psi (\nabla \psi,\nabla \psi)\theta^3 |w|^2}\\
\nonumber 
&-d^2 s^3\lambda^3 \IGT{( \pnu \psi)^3 \theta^3 |w|^2}\\
\label{I1k} 
\geq & 2d^2 \gamma^4 s^3\lambda^4 \IOT{\theta^3|w|^2} +d^2\gamma^3 s^3\lambda^3 \IGT{\theta^3 |w|^2}.
\end{align}	

On the other hand,
\begin{align*}
I_{21}=&d^2 s\Re \IOT{\Delta \varphi \overline{w} \Delta w}\\
=&-d^2 s\IOT{\Delta \varphi |\nabla w|^2}
- \frac{d^2 s}{2} \IOT{ \nabla (\Delta \varphi)\cdot \nabla |w|^2} \\
&+d^2 s\Re \IGT{\Delta \varphi \overline{w}\pnu w} \\
=&-d^2 s \IOT{\Delta \varphi |\nabla w|^2} + \dfrac{d^2s}{2}\IOT{\Delta^2 \varphi |w|^2}\\
&-\dfrac{d^2}{2}s\IGT{\pnu (\Delta \varphi) |w|^2}+d^2 s\Re \IGT{\Delta \varphi \overline{w}\pnu w}.
\end{align*}



Now, explicit computations show that
\begin{align*}
|\Delta^2\varphi|\geq C\lambda^4 \theta,\text{ in }\Omega\times (0,T),\quad \text{ and }|\pnu (\Delta \varphi)|\geq C\lambda^3\theta,\text{ on }\Gamma_1\times (0,T).
\end{align*}

According to these estimates, $I_{21}$ can be bounded in the following way
\begin{align}
\label{I21} 
\begin{split} 
I_{21}\geq &s \IOT{\left(\lambda^2 \theta |\nabla \psi|^2 +\lambda \theta \Delta \psi \right) |\nabla w|^2} -Cs\lambda^4 \IOT{\theta |w|^2}\\
&-Cs\lambda^2 \IGT{\theta |w|\pnu w}-Cs\lambda^3 \IGT{\theta |w|^2}.
\end{split}
\end{align}

On the other hand, 
\begin{align*}
I_{22}=&2d^2 s\Re \IOT{\Delta w \nabla \varphi \cdot \nabla \overline{w}}\\
=&-2d^2 s\Re \IOT{\nabla w\cdot \nabla (\nabla \varphi \cdot \nabla \overline{w})}+2d^2 s\Re \int_{0}^T \int_{\partial \Omega}(\nabla \varphi \cdot \nabla \overline{w})\pnu w d\sigma dt\\
=&K_1+K_2.
\end{align*}

On the one hand, we have 
\begin{align*}
K_1 
=&-2d^2 s \Re \IOT{\left( \nabla^2 \varphi(\nabla w,\nabla \overline{w})+\dfrac{1}{2} \nabla \varphi \cdot \nabla (|\nabla w|^2)\right)}\\
=&-2d^2 s\Re \IOT{\nabla^2 \varphi(\nabla w,\nabla \overline{w})}+d^2 s\IOT{\Delta \varphi |\nabla w|^2}\\
&-d^2 s\int_0^T \int_{\partial \Omega} \pnu \varphi |\nabla w|^2 d\sigma dt.
\end{align*}

Explicit computations show that
\begin{align*}
\nabla^2 \varphi(\nabla w,\nabla \overline{w})=-\lambda^2 \theta |\nabla w\cdot \nabla \psi|^2 -\lambda \theta \nabla^2 \psi(\nabla w,\nabla w),
\end{align*}
and
\begin{align*}
\nabla w=\pnu w \nu,\text{ on }\Gamma_0\times (0,T), \quad \pnu \varphi =-\lambda \theta \pnu \psi \text{ on }\partial \Omega\times (0,T).
\end{align*}

Then,
\begin{align*} 
K_1=&2d^2 s \Re \IOT{(\lambda^2 \theta |\nabla \psi \cdot \nabla w|^2 +\lambda \theta \nabla^2 \psi (\nabla w,\nabla \overline{w}))}\\
&-d^2 s\lambda \IOT{(\lambda |\nabla \psi|^2 +\Delta \psi)\theta |\nabla w |^2}\\
&+d^2 s\lambda \IGOT{\pnu \psi \theta |\pnu w|^2}+d^2 s\lambda \IGT{\pnu \psi \theta (|\pnu w|^2 +|\nabla_\Gamma w|^2)}.
\end{align*}

On the other hand,
\begin{align*}
K_2 
=&2d^2 s\IGOT{\pnu \varphi |\pnu w|^2}+2d^2 s\IGT{\pnu \varphi |\pnu w|^2}\\
=&-2d^2 s\lambda \IGOT{\pnu \psi \theta |\pnu w|^2}-2d^2 s\lambda \IGT{\pnu \psi \theta |\pnu w|^2}.
\end{align*}

Thus, 
\begin{align}
\label{I22} 
\begin{split} 
I_{22}=&2d^2 s\lambda \IOT{\left(\lambda \theta|\nabla \psi \cdot \nabla w|^2 + \theta \nabla^2 \psi(\nabla w,\nabla \overline{w}) \right)}\\
&-d^2 s\lambda \IOT{(\lambda |\nabla \psi|^2 +\Delta \psi)\theta |\nabla w|^2}\\
&-d^2 s\lambda \IGOT{\pnu \psi \theta |\pnu w|^2}-d^2 s\lambda \IGT{\pnu \psi \theta |\pnu w|^2}\\
&+d^2 s\lambda \IGT{\pnu \psi \theta |\nabla_\Gamma w|^2}.
\end{split} 
\end{align}

Now, $I_{23}$ is given by
\begin{align*}
I_{23}=& {\RM - } ds\Re i\IOT{\pt \varphi \overline{w}\Delta w}\\
=&{\RM - }  ds\Im \IOT{\overline{w}\nabla (\pt \varphi)\cdot \nabla w}
{\RM  + }  ds\Im \IGT{\pt \varphi \overline{w}\pnu w}. 
\end{align*}

By Young's inequality, for all $\epsilon>0$ there exists a constant $C=C(\epsilon)$ such that 
\begin{align}
\label{I23} 
\begin{split}
I_{23}\geq &-\IOT{\epsilon s\lambda \theta |\nabla w|^2 + C(\epsilon) s\lambda \theta^3 |w|^2}\\
&-\IGT{(\epsilon s\lambda \theta |\pnu w|^2 +C(\epsilon)s\lambda \theta^3 |w|^2)}.
\end{split}
\end{align}

Therefore, adding inequalities \eqref{I21}, \eqref{I22} and \eqref{I23}, choosing $\epsilon>0$ small enough and taking $\lambda_0$ and $s_0$ sufficiently large, we get
\begin{align}
\label{I2k} 
\begin{split}
\sum_{k=1}^3 I_{2k}\geq &C s\lambda \IOT{\theta |\nabla w|^2} +Cs\lambda \IGOT{\theta |\pnu w|^2}\\
& +Cs\lambda \IGT{\theta |\pnu w|^2}+d^2s\lambda \IGT{\pnu \psi \theta |\nabla_\Gamma w|^2}\\
&-C s\lambda \IOT{\theta^3 |w|^2}-Cs\lambda \IGT{\theta^3 |w|^2},
\end{split}
\end{align}
for all $\lambda \geq \lambda_0$ and for all $s\geq s_0$.

Moreover, 
\begin{align}
\label{I31} 
I_{31}=ds\Re i \IOT{\Delta \varphi \overline{w}\pt w}=-ds\Im \IOT{\Delta \varphi \overline{w}\pt w}
\end{align}

Integrating by parts, first in time and then in space, we have
\begin{align}
\nonumber 
I_{32}= 
&2ds \Im  \IOT{\pt \overline w \nabla \varphi \cdot \nabla {w}}\\
\nonumber 
=&-2ds\Im \IOT{ \overline{w}\nabla \varphi \cdot \nabla \pt w}-2ds \Im \IOT{\overline{w}\nabla \pt \varphi \cdot \nabla w}\\
\nonumber  
=&-ds \Im \IOT{\overline{w} \nabla \pt \varphi \cdot \nabla w}+ds \Im \IOT{\Delta \varphi \overline{w}\pt w}\\
\nonumber  
&-ds\Im \IGT{\pnu \varphi \overline{w}\pt w}.
\end{align} 

By Young's inequality, for all $\epsilon>0$ there exists a constant $C=C(\epsilon)$ such that
\begin{align}
\label{I32}
\begin{split} 
I_{32}\geq &-\epsilon s\lambda \IOT{\theta |\nabla w|^2}-C(\epsilon)s\lambda \IOT{\theta^3 |w|^2}\\
&+ds\Im \IOT{\Delta \varphi \overline{w}\pt w}-ds \Im \IGT{\pnu \varphi \overline{w}\pt w}.
\end{split} 
\end{align}

For the term $I_{33}$ we have:
\begin{align}
\nonumber 
I_{33}=& \Re s\IOT{\pt \varphi \overline{w}\pt w}= + \dfrac{1}{2}s  \IOT{\pt \varphi \pt (|w|^2)}\\
\label{I33} 
\geq & -Cs \IOT{\theta^3 |w|^2}.
\end{align}

Now, adding \eqref{I31}, \eqref{I32} and \eqref{I33} we conclude that 
\begin{align}
\label{I3k}
\begin{split} 
\sum_{k=1}^3 I_{3k}\geq & -\epsilon s\lambda \IOT{\theta |\nabla w|^2}-C(\epsilon)s\lambda^3 \IOT{\theta^3 |w|^2}\\
&+ds\lambda \Im \IGT{\pnu \psi \theta \overline{w}\pt w}. 
\end{split} 
\end{align}

Now, sum up inequalities \eqref{I1k}, \eqref{I2k} and \eqref{I3k}, using Young's inequality and taking $\lambda_0$ and $s_0$ large enough, we obtain
\begin{align}
\label{P1P2} 
\begin{split}
&\Re \IOT{P_1 w \overline{P_2 w}}\\
\geq & C\IOT{\left(s^3 \lambda^4 \theta^3 |w|^2 +s\lambda \theta |\nabla w|^2 + s\lambda^2 \theta |\nabla \psi \cdot \nabla w|^2 \right)}\\
&+C  \IGT{( s^3\lambda^3 \theta^3 |w|^2 +s\lambda \theta |\pnu w|^2)} -d^2 s\lambda \IGOT{\pnu \psi \theta |\pnu w|^2}\\
&+d^2 s\lambda \IGT{\pnu \psi \theta |\nabla_\Gamma w|^2}+ds\lambda \Im \IGT{\pnu \psi \theta \overline{w}\pt w}.
\end{split}
\end{align}

Notice that the global term of $\nabla_\Gamma w$ in the last inequality is not strictly positive due to $\pnu \psi<0$ on $\Gamma_1$. As we shall see in the next step, this difficulty can be avoided choosing $d$ and $\delta$ according to assumption \eqref{assumption:delta:d}. 

\noindent $\bullet$ {\bf Step 3: Estimates on $\Gamma_1\times (0,T)$.}

In this step, we devote to compute the terms defined on $\Gamma_1\times (0,T)$. By definition of $Q_1$ and $Q_2$, we have
\begin{align*}
\Re \IGT{Q_1 w \overline{Q_2w}}=\sum_{j=1}^2 \sum_{k=1}^2 J_{jk},
\end{align*}
where $J_{jk}$ stands for the $L^2(\Gamma_1\times (0,T))$-inner product between the $j^\text{th}$-term of $Q_1w$ and the $k^{th}$-term of $Q_2w$.

Firstly, notice that the term $J_{11}$ can be written as follows:
\begin{align*}
J_{11}=&-\delta ds \Re \IGT{\pnu \varphi \overline{w}\Delta_\Gamma w}\\
=&-\delta d s\lambda \IGT{\pnu \psi \theta |\nabla_\Gamma w|^2}+\delta d s \Re \IGT{\overline{w}\nabla_\Gamma(\pnu \varphi)\cdot \nabla_\Gamma w}.
\end{align*}

Using the formula $\text{div}_{\Gamma}(w\nabla_\Gamma (\pnu \psi))=\nabla_\Gamma (\pnu \psi)\cdot \nabla_\Gamma w+w\Delta_\Gamma (\pnu \psi)$ and integrating by parts we easily deduce that
\begin{align*}
&-\delta ds\lambda \Re \IGT{\overline{w}\theta \nabla_\Gamma(\pnu \psi)\cdot \nabla_\Gamma w}\\
=&\delta ds\lambda \Re \IGT{w\nabla_\Gamma (\overline{w}\theta) \nabla_\Gamma (\pnu \psi)}+\delta ds\lambda \IGT{\Delta_\Gamma (\pnu \psi)\theta |w|^2}\\
=&\dfrac{1}{2}\delta ds \lambda \IGT{\Delta_\Gamma (\pnu \psi) \theta |w|^2},
\end{align*}
where we have used that $\nabla_\Gamma \theta=0$ on $\Gamma_1\times (0,T)$.

In the same manner, $J_{12}$ can be estimate as follows:
\begin{align*}
J_{12}=&\delta s \Im  \IGT{\pt \varphi \overline{w}\Delta_\Gamma w}\\
=&-\delta s\Im \IGT{\overline{w}(\nabla_\Gamma \pt \varphi)\cdot \nabla_\Gamma w}-\delta s\Im \IGT{\pt \varphi |\nabla_\Gamma w|^2}\\ 
=&0,
\end{align*}
since $\nabla_\Gamma \psi=0$ on $\Gamma_1$.

On the other hand, the term $J_{21}$ is given by
\begin{align*}
J_{21}=&-ds\Re i \IGT{\pnu \varphi \overline{w}\pt w}
=ds\Im \IGT{\pnu \varphi \overline{w}\pt w}\\
=&-ds\lambda \Im \IGT{\pnu \psi \theta \overline{w}\pt w}.
\end{align*}

Moreover, 
\begin{align*}
J_{22}= s\Re \IGT{\pt \varphi \overline{w}\pt w}\geq -Cs\IGT{\theta^2 |w|^2}.
\end{align*}

According to the above computations, we easily deduce that 
\begin{align}
\label{Q1Q2} 
\begin{split} 
&\Re \IGT{Q_1 w \overline{Q_2 w}}\\
\geq & -\delta d s\lambda \IGT{\pnu \psi |\nabla_\Gamma w|^2}-Cs\lambda \IGT{\theta |w|^2}\\
&-Cs \IGT{\theta^2 |w|^2}-ds\lambda \Im \IGT{\pnu \psi \theta \overline{w}\pt w}.
\end{split} 
\end{align}

In the next step, we gather the terms obtained in the steps 2 and 3 to conclude the desired Carleman estimate.

\noindent $\bullet$ {\bf Step 4: Last arrangements and conclusion}

Adding the inequalities \eqref{P1P2} and \eqref{Q1Q2}, using Young's inequality, and taking $s_0$ and $\lambda_0$ large enough if it is neccesary, we obtain
\begin{align}
\label{conclu:01}
\begin{split}
&\IOT{(s^3\lambda^4 \theta^3 |w|^2 +s\lambda |\nabla w|^2 +s\lambda^2\theta  |\nabla \psi \cdot \nabla w|^2)}\\
&+\IGT{(s^3\lambda^3 |w|^2 +s\lambda |\pnu w|^2)}\\
&+ d(\delta -d)s\lambda\IGT{ |\pnu \psi | |\nabla_\Gamma w|^2}\\
\leq & C\IOT{|Pw|^2}+C\IGT{|Qw|^2}+Cs\lambda \IGOT{|\pnu w|^2}\\
&+C\IOT{|Rw|^2}+C\IGT{|R_\Gamma w|^2},
\end{split}
\end{align}
for all $\lambda \geq \lambda_0$ and $s\geq s_0$. We notice that the global term of $\nabla_\Gamma w$ on $\Gamma_1\times (0,T)$ is positive since we assumed the condition \eqref{assumption:delta:d}. Moreover, it is easy to see that

\begin{align*}
&\IOT{|Rw|^2}+ \IGT{|R_\Gamma w|^2}\\
\leq & C\IOT{\left(s^2\lambda^2\theta^2|w|^2 + |\nabla w|^2 \right)}\\
&+ C\IGT{\left(s^2\lambda^2 \theta^2|w|^2 + |\pnu w|^2 + |\nabla_\Gamma w|^2\right)}.
\end{align*}

Therefore, the last two terms of the right-hand side of \eqref{conclu:01} can be absorbed
by the left hand side taking $s_0$ and $\lambda_0$ sufficiently large. 

Finally, we come back to the original variable $v=e^{s\varphi}w$.
Taking into account that 
$$e^{-2s \varphi}  |\nabla v|^2 \leq C (s^2\lambda^2\theta^2 |w|^2 + |\nabla w|^2 )
$$
and that, for each $v \in \mathcal{V}$,
$$ 
\pnu w = e^{s \varphi} \pnu v \quad \text{ on } \Gamma_0, 
$$
we 
 obtain  \eqref{Carleman:estimate:01}. This concludes the proof of  Theorem \ref{Thm:Carleman}. 

\subsection{Proof of the Corollary \ref{Corollary:Carleman}}
\label{section:proof:corollary:carleman}

We start considering a real function $\eta \in C^\infty(\overline{\Omega})$ such that $\eta=1$ in $\overline{\Omega}\setminus \omega$ and vanishing close to $\Gamma_*$. if $w=\eta v$ in $\Omega\times (0,T)$ and $w_\Gamma=\eta v_\Gamma$ on $\Gamma_1\times (0,T)$, then it is easy to see that $(w,w_\Gamma)$ satisfies the equations
\begin{align*}
L(w)=\eta L(v)+2d\nabla \eta\cdot \nabla v+\Delta \eta v,\text{ in }\Omega\times (0,T),\quad N(w,w_\Gamma)=0,\text{ on }\Gamma_1\times (0,T),
\end{align*}
and
\begin{align*}
w=\pnu w=0, \text{ on } \Gamma_* \times (0,T),\quad w=w_\Gamma,\text{ on }\Gamma_1\times (0,T).
\end{align*}

Besides, since $v = w$ in $\Omega\setminus \omega$, we point out that
\begin{align}
\label{estimate:cor:01}
\begin{split} 
\IOT{e^{-2s\varphi} \theta |\nabla w|^2} & \geq \int_0^T \int_{\Omega\setminus \omega} e^{-2s\varphi}\theta |\nabla w|^2 dxdt \\
& = \int_0^T \int_{\Omega\setminus \omega} e^{-2s\varphi}\theta |\nabla v|^2 dxdt.
\end{split}
\end{align}

Then, applying the Carleman inequality \eqref{Carleman:estimate:01} to $w=\eta v$ and taking into account the estimate \eqref{estimate:cor:01} we easily get
\begin{align}
\label{distributed:carleman:g1}
\begin{split}
&\IOT{e^{-2s\varphi} (s^3\lambda^4 \theta^3 |v|^2 +s\lambda \theta |\nabla v|^2}\\
&+\IGT{e^{-2s\varphi}(s^3\lambda^3 \theta^3 |v|^2 +s\lambda\theta |\nabla_\Gamma v_\Gamma|^2 +s\lambda \theta |\pnu v|^2)}\\
\leq & C \IOT{e^{-2s\varphi} |L(v)|^2}+ C\IGT{e^{-2s\varphi} |N(v,v_\Gamma)|^2}\\
&+C\int_0^T\int_{\omega} e^{-2s\varphi} (s^3\lambda^4 \theta^3 |v|^2 +s\lambda \theta |\nabla v|^2)dxdt\\
&+C \IOT{e^{-2s\varphi}(|v|^2+|\nabla v|^2)},
\end{split}
\end{align}
for all $\lambda \geq \lambda_0$ and for all $s\geq s_0$. We notice that, the last two terms of the right-hand side of \eqref{distributed:carleman:g1} can be absorbed by taking $\lambda_0$ and $s_0$ sufficiently large if it is necessary. This completes the proof of the Corollary \ref{Corollary:Carleman}.

\section{Proof of the Theorem \ref{Thm:Controllability}}
\label{section:proof:thm:controllability:boundary}
We introduce the adjoint system of \eqref{intro:problem:01}
\begin{align}
\label{adjoint:z}
\begin{cases}
i\pt z+d\Delta z+ i\vec{r}\cdot \nabla z + qz=0,&\text{ in }\Omega\times (0,T),\\
i\pt z_\Gamma -d\pnu z +\delta \Delta_\Gamma z_\Gamma +i\vec{r}_\Gamma\cdot \nabla_\Gamma z_\Gamma +q_\Gamma z_\Gamma =0,&\text{ on }\Gamma_1\times (0,T),\\
z=z_\Gamma,&\text{ on }\Gamma_1\times (0,T),\\
z=0,&\text{ on }\Gamma_0\times (0,T),\\
(z,z_\Gamma)(T)=(z_T,z_{\Gamma,T}),&\text{ in }\Omega\times \Gamma_1,
\end{cases}
\end{align}
where $(q,q_\Gamma)$ are given by
\begin{align}
	\label{def:q:qg:adjoint}
	q:=i\text{div}(\vec{r}) + \overline{q}_0,\quad q_\Gamma:= i\text{div}_\Gamma (\vec{r}_\Gamma) -i\vec{r}\cdot \nu +\overline{q}_{\Gamma,0}.
\end{align}

Since $\vec{r}\in [L^\infty(0,T;W^{2,\infty}(\Omega;\mathbb{R}))]^n$ and $\vec{r}_\Gamma \in [L^\infty(0,T;W^{2,\infty}(\Gamma;\mathbb{R}))]^n$, $q$ and $q_\Gamma$ defined in \eqref{def:q:qg:adjoint} belongs to $L^\infty(0,T;W^{1,\infty}(\Omega))$ and $L^\infty(0,T;W^{1,\infty}(\Gamma_1))$, respectively, and therefore (using the change of variables $t\mapsto T-t$) problem \eqref{adjoint:z} has a unique solution $(z,z_\Gamma)\in C^0([0,T];\mathcal{V})$.  

On the other hand, according to Proposition \ref{prop:exact:controllability}, the exact controllability of \eqref{intro:problem:01} is equivalent to prove the observability inequality 
\begin{align}
\label{teo:19:01}
\|(z_T,z_{\Gamma,T})\|_\mathcal{V}^2 \leq C \int_0^T\int_{\Gamma_\ast}{|\pnu z|^2} d\sigma dt,
\end{align}
for all $(z_T,z_{\Gamma,T})\in \mathcal{V}$ and $(z,z_\Gamma)$ being the associated solution of \eqref{adjoint:z} and $\Gamma_\ast$ is given by \eqref{def:Gamma-ast}.
	Firstly, we shall introduce a cut-off function $\zeta \in C^1([0,T];\mathbb{R})$ such that 
	\begin{align*}
		0\leq \zeta \leq 1,\quad \zeta=0,\quad \forall t\in [0,T/4],\quad \zeta=1,\quad \forall t\in [3T/4,T].
	\end{align*} 
	Then, the variables $(w,w_\Gamma)=(\zeta z,\zeta z_\Gamma)$ solves the problem
	\begin{align}
		\begin{cases}
			i\pt w+d\Delta w+ i\vec{r}\cdot \nabla w+ qw= i \zeta' z,&\text{ in }\Omega\times (0,T),\\
i\pt w_\Gamma -d\pnu w +\delta \Delta_\Gamma w_\Gamma +i\vec{r}_\Gamma\cdot \nabla_\Gamma w_\Gamma +q_\Gamma w = i \zeta' z_\Gamma,&\text{ on }\Gamma_1\times (0,T),\\
w=w_\Gamma,&\text{ on }\Gamma_1\times (0,T),\\
w=0,&\text{ on }\Gamma_0\times (0,T),\\
(w,w_\Gamma)(0)=(0,0),&\text{ in }\Omega\times \Gamma_1,
		\end{cases}
	\end{align}
	By Proposition \ref{WP:prop:02}, $(w,w_\Gamma)\in C^0([0,T];\mathcal{V})$ and since $\text{supp}(\zeta')\subseteq (T/4,3T/4)$ the following inequality holds
	\begin{align*}
		\|(w, w_\Gamma)\|_{C^0([T/4,T];\mathcal{V})}^2 \leq C \|(z,z_\Gamma)\|_{L^1(T/4,3T/4;\mathcal{V})}^2.
	\end{align*} 
	In particular, we can assert that 
	\begin{align}
	\label{teo:19:02}
		\|(z_T,z_{\Gamma,T})\|_{\mathcal{V}}^2 \leq C \|(z,z_\Gamma)\|_{L^2(T/4,3T/4;\mathcal{V})}^2.
	\end{align}
	Moreover, by Carleman estimate (see Theorem \ref{Thm:Carleman}) applied to $(z,z_\Gamma)$ we obtain
	\begin{align}
	\label{teo:19:03}
	\begin{split}
		&\int_{T/4}^{3T/4}\int_\Omega e^{-2s\varphi}(s^3\lambda^4 \theta^3 |z|^2 + s\lambda^2 \theta |\nabla z|^2)dxdt\\
		&+ \int_{T/4}^{3T/4}\int_{\Gamma_1} e^{-2s\varphi}(s^3\lambda^3 \theta^3 |z_\Gamma|^2 + s\lambda \theta |\nabla_\Gamma z_\Gamma|^2)d\sigma dt\\
		\leq& C s\lambda\int_0^{T}\int_{\Gamma_*} e^{-2s\varphi}\theta |\pnu z|^2 d\sigma dt, 
	\end{split}
	\end{align}
	for all $s\geq s_0$ and $\lambda \geq \lambda_0$. Now, we fix $s$ and $\lambda$. Since $\varphi$ is bounded, there exists a constant $C>0$ such that
	\begin{align}
	\label{teo:19:04}
		1\leq C e^{-2s\varphi},\quad \forall x\in \overline{\Omega} \times (T/4,3T/4).
	\end{align}
	Now, combining \eqref{teo:19:02}, \eqref{teo:19:03} and \eqref{teo:19:04} we easily deduce \eqref{teo:19:01} and therefore the proof of the Theorem \ref{Thm:Controllability} is done.     

\section*{Acknowledgments}
The authors are indebted to the anonymous referees for
their comments and suggestions, which in particular allowed us to improve the scope 
of our problem. 


\bibliographystyle{abbrv}
\bibliography{biblio}

\end{document}